\documentclass[twoside,12pt,reqno]{amsart}
\usepackage{amssymb,latexsym,verbatim,epic,eepic,hyperref} % verbatim for ``comment''
\hypersetup{citecolor=[rgb]{1,0,0}, linkcolor=[rgb]{0,0,1}, colorlinks=true}
\theoremstyle{plain}
\newtheorem{theorem}{Theorem}
\newtheorem{lemma}[theorem]{Lemma}
\newtheorem{proposition}[theorem]{Proposition}
\newtheorem{conjecture}[theorem]{Conjecture}
\newtheorem{corollary}[theorem]{Corollary}

\theoremstyle{definition}

\theoremstyle{definition} % Bold heading, roman text
\newtheorem*{remarks}{Remarks}

\renewcommand{\arraystretch}{1.2}     % so lines in tables are not crowded
\renewcommand{\geq}{\geqslant}
\renewcommand{\leq}{\leqslant}

\begin{document}

\newcommand\BigO{\textup{O}}
\newcommand{\C}{\textup{C}}
\newcommand{\diag}{\textup{diag}}
\newcommand{\eps}{\varepsilon}
\newcommand{\F}{\mathbb{F}}
\newcommand\GL{\textup{GL}}
\newcommand{\h}[1]{{\widehat{#1}}}
\newcommand\im{\textup{im}}
\newcommand\Irr{\textup{Irr}}
\newcommand\IsfCyclic{\textsc{Is\kern-1.1pt$f$\kern-0.5ptCyclic}}
\newcommand\IsfWitness{\textsc{Is\kern-1.1pt$f$\kern-0.5ptWitness}}
\renewcommand{\l}{\lambda}
\newcommand{\M}{\textup{M}}
\newcommand{\Magma}{\text{\sc Magma}}
\newcommand{\Mat}{\textup{Mat}}
\newcommand{\MeatAxe}{\text{\sc Meat-axe}}
\newcommand\n{\newline}
\newcommand{\N}{\mathbb{N}}
\newcommand{\norm}[1]{\parallel\kern-1pt #1\kern-1pt\parallel}
\newcommand{\properdivisors}{\genfrac{}{}{0pt}{}{d|n}{1<d<n}}
\newcommand\OO{\textup{O}}
\newcommand\oo{\textup{o}}
\newcommand\ord{\textup{ord}}
\newcommand\unc{\textup{unc}}
\newcommand\Unc{\textup{Unc}}
\newcommand\U{\textup{U}}
\newcommand{\type}{\textup{type}}
\newcommand{\Z}{\mathbb{Z}}

\hyphenation{Frob-enius}
%\hyphenation{http:-www.cwu.edu/$\sim$glasbys/}

\begin{center}\large\sffamily\mdseries
Towards an efficient \MeatAxe\ algorithm using $f$-cyclic matrices:\\
the density of uncyclic matrices in $\M(n,q)$
\end{center}

\title[\tiny\upshape\rmfamily The density of uncyclic matrices]{}

%\date{Draft printed on \today}
% Begun March 2007; Continues July 2008
%\date{Submitted: November, 2008; Accepted: 2009.02.25}

\vskip11mm
\begin{center}
  To John Cannon and Derek Holt in recognition\\
  of their distinguished contributions to mathematics,\\
  and in particular, to computation and the Magma system.
\end{center}

\author{{\sffamily S.\,P. Glasby and Cheryl E. Praeger}}

\address[Glasby]{
Department of Mathematics\\   
Central Washington University\\
WA 98926-7424, USA. {\tt http://www.cwu.edu/$\sim$glasbys/}}
\address[Praeger]
{School of Mathematics and Statistics\\
35 Stirling Highway\\
University of Western Australia\\
Crawley 6009, Australia. {\tt http://www.maths.uwa.edu.au/$\sim$praeger/}}

\begin{abstract}
An element $X$ in the algebra $\M(n,\F)$ of all $n\times n$ matrices over a
field $\F$ is said to be $f$-cyclic if the underlying vector space 
considered as an $\F[X]$-module has at least one cyclic primary component.
These are the matrices considered to be ``good'' in the Holt-Rees version 
of Norton's irreducibility test in the
\MeatAxe\ algorithm. We prove that, for any finite field $\F_q$, the 
proportion of matrices in $\M(n,\F_q)$ that are ``not good'' decays
exponentially 
to zero as the dimension $n$ approaches infinity. Turning this around, we
prove that the density of ``good'' matrices in $\M(n,\F_q)$ for the
\MeatAxe\ depends on the degree, showing that it is at least
$1-\frac2q(\frac{1}{q}+\frac{1}{q^2}+\frac{2}{q^3})^n$ for $q\geq4$.
We conjecture that the density is at least
$1-\frac1q(\frac{1}{q}+\frac{1}{2q^2})^n$ for all $q$ and $n$, and 
confirm this conjecture for dimensions $n\leq 37$.
Finally we give a one-sided Monte Carlo algorithm called \IsfCyclic\  to test
whether a matrix is ``good'', at a cost of $\OO(\Mat(n)\log n)$ field
operations, where $\Mat(n)$ is an upper bound for the number of field
operations required to multiply two matrices in $\M(n,\F_q)$.
\end{abstract}

\maketitle
\centerline{\noindent 2000 Mathematics subject classification:
 15A52, 20C40}

\section{Introduction}

The \MeatAxe\ is a fundamental tool in computational representation
theory, most often used to test irreducibility of a finite matrix
group or algebra, and in the case of reducibility to construct an
invariant subspace.  A number of versions have been described in the
literature, first by R.~Parker~\cite{Parker} in 1984 and 
later by others~\cite{IL, HR, colum}. The implementations of the
\MeatAxe\ in the computer algebra systems {\sf GAP} \cite{GAP} and Magma
\cite{Magma} are based on the version of D.\,F.~Holt and
S.~Rees in~\cite{HR}. The aim of this paper is to analyse the class of
matrices used by Holt and Rees in their version of S.\,P.~Norton's
irreducibility test~\cite[Section 2]{HR}. In the language of Holt and
Rees these are matrices whose characteristic polynomials have at least
one ``good'' irreducible factor. Following~\cite{Glasby} we call them
\emph{$f$-cyclic matrices}. They are those matrices $X$ over $\F$ for which the
underlying vector space, considered as an $\F[X]$-module, has at least
one cyclic primary component (see Section~\ref{S:conj} for a detailed
definition).

Proving that the ``$f$-cyclic irreducibility test'' is a Monte Carlo
algorithm requires a lower bound on the proportion of $f$-cyclic
matrices in an irreducible subalgebra of the algebra $\M(n,q)$ of
$n\times n$ matrices over a field of order $q$. Holt and Rees derive a
lower bound sufficient for their purposes by showing that at least a
non-zero constant fraction of the matrices in such irreducible
subalgebras have a ``good'' linear factor, (see \cite[pp. 7-8]{HR}
where a lower bound of $0.234$ is proved for all $n$ and $q$).

A variant of this irreducibility test using cyclic matrices was
introduced by P.\,M.~Neumann and the second author in \cite{colum},
and analysing it required a lower bound for the proportion of cyclic
matrices in irreducible subalgebras of $\M(n,q)$.  Explicit lower
bounds were obtained of the form $1-cq^{-3}$ for the full matrix
algebra $\M(n,q)$, and similar expressions for proper irreducible
subalgebras, see \cite[Theorems 4.1 and 5.5]{NP}. Precise limiting
proportions for large $n$ are also known, see \cite{Fulman, FNP, Wall}.

In 2006 the first author began a study of $f$-cyclic
matrices, which included both a simplified proof of the $f$-cyclic
irreducibility test and also a determination of the exact proportion
of $f$-cyclic matrices in $\M(n,q)$ for small $n$.  The results
for small $n$ suggested that the proportion of $f$-cyclic matrices in
$\M(n,q)$ may admit a lower bound $1-cq^{-d(n)}$ for some constant
$c$, where $d(n)$ increases with $n$. That is, the proportion
of ``non-$f$-cyclic'' matrices may be significantly smaller than the
proportion of non-cyclic matrices. Our wish to understand how
this proportion varies as $n$ increases motivated the present
investigation. While the proportion of non-cyclic matrices in
$\M(n,q)$ is known to lie between $\frac{1}{q^2(q+1)}$ and
$\frac{1}{(q^2-1)(q-1)}$ for all $n\geq2$ by \cite[Theorem 4.1]{NP},
it turns out that the proportion of non-$f$-cyclic matrices in
$\M(n,q)$ decays to zero exponentially as $n$ increases.

\begin{theorem}\label{T:simple}
There is a positive constant $c<1$ such that, for all finite field sizes~$q$,
and all dimensions $n\geq1$, the proportion of $f$-cyclic matrices
in $\M(n,q)$ is at least $1-c^n$.
\end{theorem}

It follows from our proofs that the constant $c= 0.983$ suffices
for all $q$. Theorem~\ref{T:simple}
is proved with
$c=c(q)=\OO(q^{-1})$.  We study the class of matrices that are not
$f$-cyclic, that is to say, matrices $X\in\M(n,q)$ for which every
primary component of the underlying vector space $\F_q^n$, considered
as an $\F_q[X]$-module, is non-cyclic. We say that such matrices are
\emph{uncyclic}, and we denote by $\unc(n,q)$ the number of uncyclic
matrices in $\M(n,q)$.  A more precise version of our bounds
follows.

\begin{theorem}\label{T:uncyclic}
If $n\geq3$ and $q\geq4$, then
\[
   q^{-n-1}\left(1+\left(\frac{n-1}{2}\right)q^{-1}-q^{-3}\right)
   <\frac{\unc(n,q)}{q^{n^2}}
   < 2 q^{-1}\left(q^{-1}+q^{-2}+2q^{-3}\right)^n.
\]
The lower bound holds when $q=2,3$, and the following upper bounds hold
\[
  \frac{\unc(n,2)}{2^{n^2}}<(0.915)(0.983)^n\quad\text{and}\quad
  \frac{\unc(n,3)}{3^{n^2}}<(0.52)(0.53)^n.
\]
\end{theorem}

The upper bounds for this theorem are proved using induction on~$n$,
see Theorems~\ref{T:upperb} and~\ref{T:q=2}. Theorem~\ref{T:upperb} involves
a slightly smaller, but more elaborate, function $c^*(q)$ in place of
the constant~2, see Lemma~\ref{L:rnq}.
Our proof of the lower bound in Theorem~\ref{T:uncyclic} 
is constructive and works for all $q$, see Theorem~\ref{T:lowerb}.
We believe that the true value of
$\unc(n,q)/q^{n^2}$ is closer to the lower bound than the upper bound given in
Theorem~\ref{T:uncyclic}, and we make the following conjecture.

\begin{conjecture}\label{C} If $q\geq2$ and $n\geq1$, then\quad
$\displaystyle\frac{\unc(n,q)}{q^{n^2}}
\leq \frac{1}{q}\left(\frac{1}{q}+\frac{1}{2q^2} \right)^n.$
\end{conjecture}

A different approach to estimating $\unc(n,q)$ is to study a
probabilistic generating function for these quantities, for fixed $q$.
We introduce such a generating function in Section~\ref{S:cyc}, obtain an
infinite product expansion for it in Proposition~\ref{P:prod}, and use
it to compute the exact values of $\unc(n,q)$ as polynomials in $q$, for
small $n$. These expressions are given in Table~\ref{TableUnc} for $n\leq7$,
and are listed in an electronic database for $n\leq37$,
see~\cite[Appendix~1]{http}.
This approach enables us to verify Conjecture~\ref{C} for
$1\leq n\leq 37$, see Proposition~\ref{P:conjbound} and~\cite[Appendix~2]{http}.
%A more careful analysis of the inductive proof is given for the case $q=2$,
%leading to the upper bound in Theorem~\ref{???}, and proving
%Theorem~\ref{T:simple} in this last case.

These, to us, surprising results raise the question of whether the
improved bounds for the proportion of $f$-cyclic matrices might lead to
improvements in the \MeatAxe\ algorithm. This is a matter of ongoing
work of the authors, see \cite{GNP}. We have resolved the first issue of
whether the property of $f$-cyclicity can be identified efficiently. In
Section~\ref{S:witness} we give a Monte Carlo algorithm
that tests whether a given matrix
$X$ in $\M(n,q)$ is $f$-cyclic, and if so constructs a generator of
(possibly a direct sum of) cyclic primary summands of the underlying
space considered as an $\F_q[X]$-module. The algorithm requires
$\OO(\Mat(n)\log n)$ field operations, where $\Mat(n)$ is an upper bound
for the number of field operations required to multiply two matrices in
$\M(n,q)$, and the construction of a constant number (depending on the
desired failure probability) of random vectors in $\F_q^n$. For a
precise statement see Theorem~\ref{T:cost}.

Section~\ref{S:conj} gives a (known) formula for the size $|X^{\GL(n,q)}|$
of the $\GL(n,q)$-orbit containing $X\in\M(n,q)$ (with $\GL(n,q)$ acting by 
conjugation). The formula depends on the
Frobenius canonical form of $X$
which, in turn, depends on certain partitions. We define notation,
and introduce an invariant of the $\GL(n,q)$-orbit called the {\it type} of $X$.
In Section~\ref{S:cyc} the generating function
$\sum_{n\geq0}\frac{\unc(n,q)}{|\GL(n,q)|}u^n$ is expressed as an infinite product.
The infinite product gives rise to a formula for $\unc(n,q)$ involving
sums over certain partitions of rational functions in $q$. It not obvious from
the formula that $\unc(n,q)$ is a polynomial in $q$ with integer coefficients.
Although the formula is explicit, we were unable to use it to prove upper bounds
or lower bounds for $\unc(n,q)$. In Section~\ref{S:lowerb} we show that
$\unc(n,q)$ is at least
$q^{n^2-n-1}+\frac{n}{2}q^{n^2-n-2}+\OO(q^{n^2-n-3})$
by counting the number of matrices in certain large classes
of uncyclic matrices. Finding upper bounds in Section~\ref{S:upperb} (for
$q>2$) and in Section~\ref{S:2upperb} (for $q=2$) involved a rather
sensitive mathematical induction. The final Section~\ref{S:witness} gives a
practical Monte Carlo $\OO(\Mat(n)\log n)$ algorithm to test whether a
given matrix $X$ is $f$-cyclic relative to some irreducible divisor of $c_X(t)$.
This algorithm avoids the expensive step of evaluating a divisor of
$c_X(t)$ at $X$. Moreover, it outputs a (witness) vector~$u$ which
can be used when applying Norton's irreducibility test~\cite[Section 2.1]{HR}.

\section{Conjugacy Classes in $\GL(n,q)$}\label{S:conj}

A partition of $n\in\N:=\{0,1,2,\dots\}$, written
$\lambda\vdash n$, is an unordered sum $n=\sum_{i\geq1}\lambda_i$
where the parts $\lambda_i$ lie in $\N$. A partition can be represented
by (a)~its parts, (b)~its Young (or Ferrers) diagram \cite{StanleyII}, or
(c)~by the multiplicities of its parts. We write
$\lambda=(\lambda_1,\lambda_2,\dots)$ where $\lambda_1\geq\lambda_2\geq\cdots$
and $n=\sum_{i\geq1}\lambda_i$. Set $|\lambda|:=\sum_{i\geq1}\lambda_i$.
It is convenient to abbreviate a partition by omitting all (or some) of the
trailing zeroes. We shall commonly write
$\lambda=(\lambda_1,\lambda_2,\dots,\lambda_k)$ where
$\lambda_1\geq\lambda_2\geq\cdots\lambda_k>0$ and
$\lambda_{k+1}=\lambda_{k+2}=\cdots=0$. The empty partition,
or partition of zero, is written $(0,0,\dots)$ or simply $()$.

The Young diagram of $\lambda=(\lambda_1,\dots,\lambda_k)$ is a rectangular 
array of $|\lambda|$ boxes arranged in $k$ left-justified rows, with 
$\lambda_i$ boxes in row $i$, for each $i$. For example, Figure~\ref{F:Young}
shows the Young diagrams for the partitions $\lambda=(5,3,3,1)$ and
$\mu=(4,3,3,1,1)$ of $n=12$.
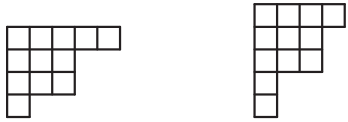
\begin{figure}[!ht]
  \setlength{\unitlength}{3mm}
  \begin{picture}(5,4)
    \thicklines
    \drawline(0,0)(0,4)\drawline(1,0)(1,4)\drawline(2,1)(2,4)\drawline(3,1)(3,4)
    \drawline(4,3)(4,4)\drawline(5,3)(5,4)
    \drawline(0,0)(1,0)\drawline(0,1)(3,1)\drawline(0,2)(3,2)\drawline(0,3)(5,3)
    \drawline(0,4)(5,4)
  \end{picture}
  \qquad\qquad
  \begin{picture}(4,5)
    \thicklines
    \drawline(0,0)(0,5)\drawline(1,0)(1,5)\drawline(2,2)(2,5)\drawline(3,2)(3,5)
    \drawline(4,4)(4,5)
    \drawline(0,0)(1,0)\drawline(0,1)(1,1)\drawline(0,2)(3,2)\drawline(0,3)(3,3)
    \drawline(0,4)(4,4)\drawline(0,5)(4,5)
  \end{picture}
  \caption{Young diagrams for $\lambda=(5,3,3,1)$ (left) and
  $\mu=(4,3,3,1,1)$ (right).}\label{F:Young}
\end{figure}
\vskip-2mm % not sure why a negative vskip is needed

By interchanging the rows and columns of the Young diagram of $\lambda$,
we obtain the Young diagram of another partition,
called the {\it conjugate} partition, and denoted $\lambda'$. For example, 
in Figure~\ref{F:Young}, $\lambda'=\mu$ and $\mu'=\lambda$. The number
of parts of $\lambda$ equal to $i$, that is to say, the multiplicity of $i$, is denoted
$m_i(\lambda)$ or simply $m_i$. We occasionally write
$\lambda=1^{m_1}2^{m_2}3^{m_3}\cdots$. The number of non-zero parts of $\lambda$,
written $\ell_1(\lambda)$, is the number of squares in the first
column of the Young diagram of $\lambda$. More generally,
$\ell_i(\lambda)$ denotes the number of squares in the first $i$
columns of the Young diagram of $\lambda$.

The vector $m(\lambda)=(m_1(\lambda),m_2(\lambda),\dots)$ in $\N^\infty$
need not be a partition because the coordinates need not satisfy
$m_i(\lambda)\geq m_{i+1}(\lambda)$ for $i\geq1$.
Denote by $m(\lambda)^\oo$ the partition obtained from $m(\lambda)$ by permuting
the coordinates so that they are weakly decreasing. The formula for the order
$|\C_{\GL(n,q)}(X)|$ of the
centralizer of an element $X\in\M(n,q)$ involves three vectors:
$m(\lambda)$, $\ell(\lambda):=(\ell_1(\lambda),\ell_2(\lambda),\dots)$,
and $e(\lambda):=(m(\lambda)^\oo)'$, for various partitions $\lambda$, see
(\ref{E:cent}) and (\ref{E:cq}) below. As an example, if $\lambda=(5,3,3,1)$, then
$\lambda'=(4,3,3,1,1)$, and
\[
  m(\lambda)=(1,0,2,0,1,0,\dots),\quad
  \ell(\lambda)=(4,7,10,11,12,12,\dots),\quad
  e(\lambda)=(3,1,0,\dots).
\]
The reader should not confuse the vector $e(\lambda)$ with the
symmetric polynomial $e_\lambda$ defined in ~\cite[p.~290]{StanleyII}.
It is convenient to define the dot product $x\cdot y:=\sum_{i\geq1}x_iy_i$ of
vectors $x,y\in\N^\infty$ in the case that the sum is finite, for example, when
$x$ or $y$ has finite support. Also define $\norm{x}^2:=x\cdot x$.

\begin{lemma}\label{L1} Let $\lambda$ be a partition of $|\lambda|$. Then
\begin{enumerate}
\item[(a)] $|\lambda|=\sum_{i\geq1}i\,m_i(\lambda)=|\lambda'|$,
\item[(b)] $m_i(\lambda)=\lambda'_i-\lambda'_{i+1}$,
\item[(c)] $\ell_i(\lambda)=\lambda'_1+\cdots+\lambda'_i=\left(\sum_{k<i}
  km_k(\lambda)\right)+i\left(\sum_{k\geq i}m_k(\lambda)\right)$,
\item[(d)] $m(\lambda)\cdot\ell(\lambda)
  =\norm{\lambda'}^2\equiv |\lambda|\pmod2$,
\item[(e)] $e_k(\lambda)=|\{i\mid m_i(\lambda)\geq k\}|$,
\item[(f)] $\norm{\lambda'}^2\geq|\lambda|$ with equality if and only if
  $\lambda=(|\lambda|,0,0,\dots)$,
\item[(g)] $|e(\lambda)|=\sum_{i\geq1}m_i(\lambda)=\lambda'_1$. 
\end{enumerate}
\end{lemma}

\begin{proof} The proofs of parts (a),(b) are elementary, see
\cite[p.~287]{StanleyII}. Counting the squares in the first $i$ columns 
of the Young diagram for $\lambda$ by columns gives the first formula
for $\ell_i(\lambda)$ in part (c), while counting by rows gives
the second. Consider part~(d):
\begin{align*}
  m(\lambda)\cdot\ell(\lambda)
  &=\sum_{i\geq1}m_i(\lambda)\ell_i(\lambda)\\
  &=(\lambda'_1-\lambda'_2)\ell_1(\lambda)
    +\sum_{i\geq2}(\lambda'_i-\lambda'_{i+1})
    \ell_i(\lambda)&&\quad\text{by part~(b)}\\
  &=(\lambda'_1-\lambda'_2)\lambda'_1
    +\sum_{i\geq2}\left(\lambda'_i\ell_{i-1}(\lambda)+(\lambda'_i)^2
    -\lambda'_{i+1}\ell_i(\lambda)\right) 
    &&\quad\text{as $\ell_i(\lambda)=\ell_{i-1}(\lambda)+\lambda'_i$}\\
  &=-\lambda'_2\lambda'_1+\lambda'_2\lambda'_1+\sum_{i\geq1}(\lambda'_i)^2
    =\norm{\lambda'}^2.
\end{align*}
However, $(\lambda'_i)^2\equiv\lambda'_i\pmod 2$ and so
$\norm{\lambda'}^2\equiv\sum_{i\geq1}\lambda'_i\pmod 2$. Part~(d)
now follows as $|\lambda'|=|\lambda|$.
Part~(e) follows from the elementary fact
$\lambda'_k=|\{i\mid \lambda_i\geq k\}|$, while part~(f) follows from
(d) and the observation that
${\lambda'_i}^2\geq\lambda'_i$ with equality if and only if $\lambda'_i=0,1$.
Finally, part~(g) follows as $\sum_{i\geq1}m_i(\lambda)$ and $\lambda'_1$
both count the number of rows in the Young diagram of $\lambda$, and
$e(\lambda)=(m(\lambda)^\oo)'$ so
$\sum_{i\geq1}e_i(\lambda)=\sum_{i\geq1}m_i(\lambda)$.
\end{proof}

Recall that $\M(n,q)$ is the algebra of $n\times n$ matrices over $\F_q$,
and let $G=\GL(n,q)$ denote the general linear group, its group of units.
A formula for the size $|X^G|$ of the $G$-orbit of a matrix
$X\in\M(n,q)$ dates back at least to~\cite{Kung,Stong}. Our formula
is better suited for calculation.
Clearly, $|X^G|=|G:C_G(X)|$ and the
structure of the centralizer $C_G(X)$ of $X$ depends on the Frobenius
(or rational) canonical form of $X$. 
Suppose that the characteristic polynomial $c_X(t)$ factors as
$\prod_f f^{\nu(f)}$ where the product is over monic irreducible
polynomials $f(t)\in\F_q[t]$, and $\nu(f)\in\N$ (possibly $\nu(f)=0$). The structure of
$C_G(X)$ depends on partitions $\lambda(f,X)$ of $\nu(f)$
which we abbreviate $\lambda(f)$ when the dependence on $X$ is clear,
see~\cite{Kung,Stong}.
The vector space $V=\F_q^{1\times n}$ is an $\F_q[X]$-module, and
$V(f)=\ker f^{\nu(f)}(X)=\ker f(X)^{\lambda(f)_1}$ is its $f$-primary component.
Let $X(f)$ denote the restriction of $X$ to $V(f)$. 
Thus the minimal polynomial of $X(f)$ is $f^{\lambda(f)_1}$, and that of
$X$ is $m_X(t)=\prod_f f^{\lambda(f)_1}$. Now $X$ is conjugate
to a block diagonal matrix $\bigoplus X(f)$ and $V(f)$ is isomorphic
as an $\F_q[X]$-module to
\[
  V(f)\cong\bigoplus_{i\geq1}\F_q[t]/(f(t)^{\lambda(f)_i}).
\]

Two matrices $X$ and $Y$ lie in the same $G$-orbit if and only if they have the
same Frobenius canonical form, that is, if and only if $\lambda(f,X)=\lambda(f,Y)$ for all
monic irreducibles $f$. It is convenient to define a formal expression
called the {\it type} of $X$ written $\type(X):=\prod_f f^{\lambda(f,X)}$.
Two formal expressions of this kind are regarded as equal if and only if 
their respective exponent partitions are equal.
Thus $X$ and $Y$ lie in the same $G$-orbit if and only if $\type(X)=\type(Y)$.
As it is sometimes convenient to omit trivial factors $f^0$ from
the product $c_X(t)=\prod_f f^{|\lambda(f,X)|}$,  it is therefore
sometimes convenient to omit factors $f^{(0,0,\dots)}$ from $\type(X)$.

It follows from~\cite{Kung,Stong} that
\begin{equation}\label{E:cent}
  |\C_{\GL(n,q)}(X)|=\prod_f |\C_{\GL(V(f))}(X(f))|=\prod_f c(\lambda(f),q^{d(f)})
\end{equation}
where $d(f):=\deg(f)$ and $c(\lambda,q)$ is the function
\[
  c(\lambda,q):=\prod_{i=1}^{\lambda_1}\prod_{k=1}^{m_i(\lambda)}
    (q^{\ell_i(\lambda)}-q^{\ell_i(\lambda)-k})=q^{m(\lambda)\cdot\ell(\lambda)}
    \prod_{i=1}^{\lambda_1}\prod_{k=1}^{m_i(\lambda)}(1-q^{-k}),
\]
see~\cite{Kung,Stong}. By Lemma~\ref{L1}(d) and (e), $c(\lambda,q)$ may be rewritten as
\begin{equation}\label{E:cq}
  c(\lambda,q)=q^{\norm{\lambda'}^2}\prod_{k\geq1}(1-q^{-k})^{e_k(\lambda)}.
\end{equation}
In summary, 
\begin{equation}\label{E:Orbit}
  |X^{\GL(n,q)}|=|\GL(n,q)|\prod_f \frac{1}{c(\lambda(f),q^{d(f)})}.
\end{equation}

The following table of values of $c(\lambda,q)$ both illustrates
formula~(\ref{E:cq}), and provides data for the proof of Lemma~\ref{L:rnq1}.
In this table we shall assume $\lambda_1>\lambda_2>\lambda_3$, and
we use the notation $1^{m_1}2^{m_2}\cdots$ to
indicate multiplicities $m(\lambda)=(m_1,m_2,\dots)$. For example,
$(\lambda_1,\lambda_2)$ is written as $\lambda_1^1\lambda_2^1$ because
$\lambda_1$ and $\lambda_2$ each occur once,
given our assumption $\lambda_1>\lambda_2$.

\begin{table}[ht]
  \begin{center}
  \begin{tabular}{|c|c|c|c|c|}\hline
  $\lambda$&$|\lambda|$&$\lambda'$&$e(\lambda)$&$c(\lambda,q)$\\ \hline \hline
  $\lambda_1^1\lambda_2^1$&$\lambda_1+\lambda_2$
    &$1^{\lambda_1-\lambda_2}2^{\lambda_2}$&$(2)$
    &$q^{|\lambda|+2\lambda_2}(1-q^{-1})^2$\\ \hline
  $\lambda_1^2$&$2\lambda_1$&$2^{\lambda_1}$&$(1,1)$
    &$q^{2|\lambda|}(1-q^{-1})(1-q^{-2})$\\ \hline \hline
  $\lambda_1^1\lambda_2^1\lambda_3^1$&$\lambda_1+\lambda_2+\lambda_3$
    &$1^{\lambda_1-\lambda_2}2^{\lambda_2-\lambda_3}3^{\lambda_3}$&$(3)$
    &$q^{|\lambda|+2\lambda_2+4\lambda_3}(1-q^{-1})^3$\\ \hline
  $\lambda_1^1\lambda_2^2$&$\lambda_1+2\lambda_2$
    &$1^{\lambda_1-\lambda_2}3^{\lambda_2}$&$(2,1)$
    &$q^{|\lambda|+6\lambda_2}(1-q^{-1})^2(1-q^{-2})$\\ \hline
  $\lambda_1^2\lambda_2^1$&$2\lambda_1+\lambda_2$
    &$2^{\lambda_1-\lambda_2}3^{\lambda_2}$&$(2,1)$
    &$q^{2|\lambda|+3\lambda_2}(1-q^{-1})^2(1-q^{-2})$\\ \hline
  $\lambda_1^k$&$k\lambda_1$
    &$k^{\lambda_1}$&$1^k$
    &$q^{\lambda_1 k^2}\prod_{i=1}^k(1-q^{-i})$\\ \hline
  \end{tabular}
  \vskip1mm
  \caption{Values of $c(\lambda,q)$.}\label{Tableclq}
  \end{center}
\end{table}

For a monic irreducible polynomial $g$ over $\F$, a matrix $X\in\M(n,\F)$ 
is said to be \emph{$f$-cyclic relative to $g$} if
the restriction $X(g)$
of $X$ to the $g$-primary component $V(g)$ of $V=\F^{1\times n}$ is cyclic.
Although we are interested to count matrices $X$ that are $f$-cyclic relative
to some monic irreducible divisor $g$ of $c_X(t)$, the complementary count is easier.
We call $X$ {\it uncyclic} if $X(g)$ is not cyclic for {\it all} monic irreducible
divisors $g$ of $c_X(t)$. Equivalently, $X$ is uncyclic if and only if
$\lambda(g)'_1\ne1$ for all $g$ (that is, $\lambda(g)$ has zero or at least two
parts for each $g$). One can readily see from the factorizations
$c_X(t)=\prod g^{\nu(g)}$ and $m_X(t)=\prod g^{\mu(g)}$ of the characteristic
and minimal polynomials of $X$ whether or not $X$ is $f$-cyclic (or uncyclic):
$f$-cyclic relative to $g$ means $\nu(g)=\mu(g)$, and uncyclic means that,  
for all $g$, $\nu(g)>\mu(g)>0$ or
$\nu(g)=\mu(g)=0$.

\section{Generating function as an infinite product}\label{S:cyc}

In this section we express the generating function
\begin{equation}\label{E:gen}
  \Unc_q(u):=1+\sum_{n=1}^\infty \frac{\unc(n,q)}{|\GL(n,q)|}\,u^n
\end{equation}
as an infinite product. It is more convenient to consider the weighted
proportion $\frac{\unc(n,q)}{|\GL(n,q)|}$ of uncyclic matrices in $\M(n,q)$
because orbit sizes have a factor $|\GL(n,q)|$ in the numerator.

Our main tool is the cycle index for $\M(n,q)$ which is defined as
\begin{equation}\label{E:cycidx}
  Z_{\M(n,q)}:=\frac{1}{|\GL(n,q)|}\sum_{X\in\M(n,q)}
  \left(\prod_f x_{f,\lambda(X,f)}\right)
\end{equation}
where the product is over all monic irreducible polynomials and the 
$x_{f,\lambda}$ are indeterminates, see~\cite{Kung,Stong} and
\cite[pp.~35-36]{J}.
If we set $x_{f,()}:=1$ for each $f$, then for each $X$ the
product in~(\ref{E:cycidx}) has finitely many factors different to~1.

Stong~\cite{Stong}, building on the work of Kung~\cite{Kung}, proves that
\begin{equation}\label{E:kung}
  1+\sum_{n=1}^\infty Z_{\M(n,q)}u^n=\prod_f
  \left(\sum_{\l}x_{f,\l} \frac{u^{|\l|d(f)}}{c(\l,q^{d(f)})}\right)
\end{equation}
where the sum on the right-hand side is over all partitions
$(),(1),(2),(1,1),\dots$. By convention
\[
  c((),q)=|\GL(0,q)|=\unc(0,q)=1.
\]

\begin{proposition}\label{P:prod} Let $\Lambda_1$ be the set of 
partitions $\l$ such that $\lambda'_1\neq1$ (equivalently $\lambda$
has $0$ or at least $2$ parts). Then 
\begin{equation}\label{E:prod}
  \Unc_q(u)=\sum_{n\geq0}\frac{\unc(n,q)}{|\GL(n,q)|} u^n=\prod_f
    \left(\sum_{\l\in\Lambda_1}\frac{u^{|\l(f)|d(f)}}{c(\l(f),q^{d(f)})}\right).
\end{equation}
\end{proposition}

\begin{proof} From the remarks above, $X$ is uncyclic
if and only if $\lambda(f)\in\Lambda_1$ for all $f$. As the set of
uncyclic matrices in $\M(n,q)$ is a union of $\GL(n,q)$-orbits, it
follows from (\ref{E:Orbit}) that
\[
  \unc(n,q)=\sum |\GL(n,q)|\prod_f \frac{1}{c(\lambda(f),q^{d(f)})}
\]
where the sum ranges over all decompositions $n=\sum |\lambda(f)|d(f)$
with $\lambda(f)\in\Lambda_1$. This proves (\ref{E:prod}).

An alternative proof uses (\ref{E:kung}). In (\ref{E:cycidx}) set
$x_{f,\lambda}=1$ if $\lambda\in\Lambda_1$, and 0 otherwise. Then
$Z_{\M(n,q)}$ equals $\unc(n,q)/|\GL(n,q)|$. On the other hand,
the bracketed sums of~(\ref{E:kung}) and (\ref{E:prod}) are equal.
\end{proof}

As the bracketed sum in (\ref{E:prod}) is the same for all $f$ with
degree $r$, we define
\begin{equation}\label{E:Aa}
  A(q,u):=\sum_{\lambda\in\Lambda_1}\frac{u^{|\lambda|}}{c(\lambda,q)}
  \quad\text{and}\quad
  a_n(q):=\sum_{\lambda\vdash n,\, \lambda\ne(n)}\frac{1}{c(\lambda,q)}.
\end{equation}
Thus $A(q,u)=\sum_{n\geq0}a_n(q)u^n$ where $a_0(q)=1$, $a_1(q)=0$,
$a_2(q)=|\GL(2,q)|^{-1}$, etc. Denote by $N(r,q)$ the number of monic
irreducible polynomials over $\F_q$ of degree~$r$.
%\[
%  N(r,q)=\sum_{s|r}\mu(\frac{r}{s})q^s.
%\]
Then (\ref{E:prod}) may be rewritten
\begin{equation}\label{E:Nprod}
  \Unc_q(u)=\sum_{n\geq0}\frac{\unc(n,q)}{|\GL(n,q)|} u^n
  =\prod_{r\geq1} A(q^r,u^r)^{N(r,q)}=
  \prod_{r\geq1} \left(1+\sum_{n\geq2}a_n(q^r)u^{rn}\right)^{N(r,q)}.
\end{equation}

A closed formula for $\unc(n,q)$ can be obtained by expanding the
products in (\ref{E:Nprod}). This formula, though unwieldy, may be used to
to determine $\unc(n,q)$ for small $n$.

\begin{lemma}\label{L:multi}
Given $n\in\N$ and a partition $\l=1^{m_1}2^{m_2}\cdots$ with $\l'_1\leq n$,
denote the multinomial coefficient
$\binom{n}{n-\sum_{i\geq1} m_i,m_1,m_2,\dots}=
\frac{n!}{(n-\lambda'_1)!m_1!m_2!\cdots}$
by $\binom{n}{m(\l)}$. Then
\begin{equation}\label{E:multi}
  (1+a_1u+a_2u^2+\cdots)^n=
    \sum_{k\geq0}\left(\sum_{\l\vdash\, k}
    \binom{n}{m(\l)}a^{m(\l)}\right)u^k
\end{equation}
where $a^{m(\l)}:=a_1^{m_1}a_2^{m_2}\cdots$.
\end{lemma}

\begin{proof}
Set $a_0:=1$. Expanding the left-hand side of (\ref{E:multi}) gives
\begin{equation}\label{E:exp}
  \sum_{\l\in\N^n}a_{\l_1}u^{\l_1}a_{\l_2}u^{\l_2}\cdots a_{\l_n}u^{\l_n}=
  \sum_{k\geq0}
  \left(\sum_{\l\in\N^n,\,  {|\l|=k}}a_{\l_1}a_{\l_2}\cdots a_{\l_n}\right)u^k.
\end{equation}
The term $a_{\l_1}\cdots a_{\l_n}$ will be repeated $\binom{n}{m(\l)}$
times, where $\binom{n}{m(\l)}$ is the number of distinct elements of $\N^n$
obtained by permuting the coordinates of $\l=(\l_1,\dots,\l_n)$. If
$1^{m_1}2^{m_2}\cdots$ is the unique partition corresponding to $\l$, then
$a_{\l_1}\cdots a_{\l_n}=a^{m(\l)}$ because $a_i$ has multiplicity
$m_i$ for $i\geq1$, and multiplicity $n-\sum_{i\geq1} m_i=n-\l'_1$ for $i=0$
by Lemma~\ref{L1}(g).
\end{proof}

Lemma~\ref{L:multi} may be used to expand the powers in (\ref{E:Nprod}).
Since in (\ref{E:Nprod}) we have $a_1=0$, it follows from (\ref{E:exp}) that
the inner sum in (\ref{E:multi})
is over partitions $\lambda$ of $k$ with no part of size~1.
For example, if $k=5$, then
$\lambda=(5)$ or $(3,2)$ and $\binom{n}{m(\l)}$ equals $n$ or $n(n-1)$,
respectively.
Expanding the power $(1+a_2z^2+a_3z^3+\cdots)^n$ using Lemma~\ref{L:multi} gives
\begin{align*}\label{E:genpow}
  &1+na_2z^2+na_3z^3
    +\left(na_4+\binom{n}{2}a_2^2\right)z^4
    +\left(na_5+2\binom{n}{2}a_2a_3\right)z^5+\cdots\\
  &=1+n\left(\sum_{i} a_iz^i\right)
  +\binom{n}{2}\left(\sum_{i} a_i^2z^{2i}+2\sum_{i<j} a_ia_jz^{i+j}\right)\\
  &\quad+\binom{n}{3}\left(\sum_{i} a_i^3z^{3i}+3\sum_{i<j} a_i^2a_jz^{2i+j}
    +3\sum_{i<j} a_ia_j^2z^{i+2j}+6\sum_{i<j<k} a_ia_ja_kz^{i+j+k}\right)
  +\cdots.
\end{align*}
In order to evaluate (\ref{E:Nprod}) it is useful to substitute $z=u^r$ and
$n=N(r,q)$ in the above expression. By using (\ref{E:multi}) and
(\ref{E:Nprod}) one can, in principle, write down a closed form for
$\unc(n,q)$. The resulting closed form is rather complicated, and it is
not obviously useful for bounding $\unc(n,q)$.
In \cite[Appendix~2]{http} we give a {\sc Magma} \cite{Magma} computer
program for computing $\unc(n,q)$
for small $n$. Given that the number of partitions of $n$ (even those
with no part of size~1) is asymptotically
exponential %$p(n)\sim\frac{1}{4n\sqrt3}\textup{exp}(\pi\sqrt{\frac{2n}{3}})$
(see \cite[p.~70]{Andrews}), our computer program can compute $\unc(n,q)$ only
for small $n$.

For very small values of $n$ one does not need a computer program.
Equating the coefficient of $u^n$ for $n\leq5$ on both sides of
(\ref{E:Nprod}) gives values of $\frac{\unc(n,q)}{|\GL(n,q)|}$ in
terms of the polynomials $a_n(q)$ defined in (\ref{E:Aa}). This information
is summarized in Table~\ref{TableSplit}.
\def\phan{\phantom{|_{|}}} % add space to first row
\begin{table}[ht]
  \begin{center}\renewcommand{\arraystretch}{1.5}
  \begin{tabular}{|c|l|l|}\hline
    $n$&${\frac{\unc(n,q)}{|\GL(n,q)|}}_{\phan}$&$a_n(q)$\\ \hline
    2&$\binom{q}{1}a_2(q)$&$\frac{1}{c((1,1),q)}=\frac{1}{q^4-q^3-q^2+q}$\\
    3&$\binom{q}{1}a_3(q)$
     &$\frac{1}{c((1,1,1),q)}+\frac{1}{c((2,1),q)}
     =\frac{q^3 + q^2 - 1}{q^8 - q^7 - q^6 + q^4 + q^3 - q^2}$\\
    4&$\binom{q}{1}a_4(q)+\binom{q}{2}a_2(q)+N(2,q)a_2(q^2)$
     &$\frac{q^7+q^6+q^5-q^4-q^3-q^2+1}{q^{13}-q^{12}-q^{11}+2q^8-q^5-q^4+q^3}$\\
    5&$\binom{q}{1}a_5(q)+q(q-1)a_2(q)a_3(q)$
     &$\frac{q^{12}+q^{11}+q^{10}-q^8-2q^7-q^6+q^4+q^3+q^2-1}{q^{19}-q^{18}
     -q^{17}+q^{14}+q^{13}+q^{12}-q^{11}-q^{10}-q^9+q^6+q^5-q^4}$\\
%    6&$\binom{q}{1}a_6(q)+2\binom{q}{2}a_2(q)a_4(q)
%    + \binom{q}{3}a_2(q)^3 +\binom{q}{2}a_3(q)^2+N(2,q)a_3(q^2)$\\
%     &\hfill$+N(3,q)a_2(q^3)
%    +N(1,q)N(2,q)a_2(q)a_2(q^2)$\\
    \hline
  \end{tabular}
\vskip1mm
  \caption{Values of $\frac{\unc(n,q)}{|\GL(n,q)|}$ and $a_n(q)$
  for $2\leq n\leq5$.}\label{TableSplit}
  \end{center}
\end{table}

It is easy to show that $\unc(1,q)=0$. The values of $\unc(n,q)$ for
$n=2,3,4,5$ can be computed from Table~\ref{TableSplit}.
We list the values and $\unc(n,q)$ for $n\leq7$ in Table~\ref{TableUnc} below. 

\begin{table}[ht]
  \renewcommand{\arraystretch}{1.05}    % so lines in tables are not crowded
  \begin{center}%~\label{TableUnc}
  \begin{tabular}{|l|l|}\hline
  $n$&$\unc(n,q)$\\ \hline
  2&$q$\\
  3&$q^5 + q^4 - q^2$\\
  4&$q^{11}+2q^{10} - 2q^7 - q^5 + q^4$\\
  5&$q^{19}+2q^{18}+2q^{17}+q^{16}-q^{15}-2q^{14}-3q^{13}-q^{12}+q^{10}
      +q^9+q^8-q^7$\\
  6&$q^{29}+3q^{28}+3q^{27}+3q^{26}-q^{25}-5q^{23}-5q^{22}-3q^{21}-2q^{20}
      +2q^{18}+4q^{17}+3q^{15}$\\
   &$\hfill-q^{14}-2q^{12}+q^{11}$\\
7&$q^{41}+3q^{40}+5q^{39}+5q^{38}+3q^{37}-4q^{35}-9q^{34}-11q^{33}-
    12q^{32}-7q^{31}-3q^{30}+4q^{29}$\\
 &$\hfill+6q^{28}+11q^{27}+8q^{26}+7q^{25}
    +q^{23}-3q^{22}-2q^{21}-3q^{20}+2q^{17}-q^{16}$\\
  \hline
  \end{tabular}
  \vskip2mm
  \caption{Values of $\unc(n,q)$ for $2\leq n\leq7$.}\label{TableUnc}
  \end{center}
\end{table}

The polynomials $\unc(n,q)$ for $n\leq 37$ were computed with the
\Magma\ \cite{Magma} programs in~\cite[Appendix~2]{http} and stored in the
database~\cite[Appendix~1]{http}.
Lemma~\ref{L:positive} below is useful for bounding polynomials in
$q$ (or $q^{-1}$).

\begin{lemma}\label{L:positive}
Suppose that $m,n$ are positive integers and
$\alpha_0,\alpha_1,\dots,\alpha_{m-1}$,
$\beta_0,\beta_1,\dots,\beta_{n-1}$ are non-negative real numbers. Set
\[
  c(q):=(\alpha_{m-1}q^{m-1}+\cdots+\alpha_1 q+\alpha_0)q^n
  -(\beta_{n-1}q^{n-1}+\cdots+\beta_1 q+\beta_0).
\]
%Then $c(q)$ is strictly increasing for $q>q_0$.
If $q_0\geq0$ and $c(q_0)\geq0$, then $c(q)\geq0$ for all $q\geq q_0$.
\end{lemma}

\begin{proof}
Set $a(q):=\alpha_{m-1}q^{m-1}+\cdots+\alpha_1 q+\alpha_0$, and
$b(q):=\beta_{n-1}q^{-1}+\cdots+\beta_1 q^{-(n-1)}+\beta_0q^{-n}$.
Then $c(q)=(a(q)-b(q))q^n$. Since $a(q)\geq a(q_0)$ and $b(q_0)\geq b(q)$,
it follows that $a(q)-b(q)\geq a(q_0)-b(q_0)$ and so $c(q)\geq c(q_0)\geq 0$.
Thus $c(q)\geq0$ for all $q\geq q_0$.
\end{proof}

Lemma~\ref{L:positive} may be applied to 
verify Conjecture~\ref{C} for small $n$.

%\begin{conjecture}\label{C}
%  If $n\geq 1$ and $q\geq2$, then
%  $\unc(n,q)\leq q^{n^2-n-1}\left(1+\frac1{2q}\right)^n$. 
%\end{conjecture}

\begin{proposition}\label{P:conjbound}
If $q\geq2$ and $1\leq n\leq37$, then
$\unc(n,q)\leq q^{n^2-n-1}(1+\frac{1}{2q})^n$.
\end{proposition}

\begin{proof}
The idea is to list the difference polynomials
$d_n(q)=q^{n^2-n-1}(1+\frac{1}{2q})^n-\unc(n,q)$ for $1\leq n\leq37$ and repeatedly
apply Lemma~\ref{L:positive}. For example, $d_5(q)$ equals
%when $n=5$ the difference polynomial is
\[
  d_5(q)=\frac{1}{2}q^{18}+\frac{1}{2}q^{17}+\frac{1}{4}q^{16}+\frac{21}{16}q^{15}
  +\frac{65}{32}q^{14}+3q^{13}+q^{12}-q^{10}-q^9-q^8+q^7,
\]
and Lemma~\ref{L:positive} with $q_0=2$ shows that polynomials
$d_5(q)-q^7$ and $q^7$ are both non-negative for $q\geq2$. Adding shows
$d_5(q)\geq0$ for $q\geq2$. For more a complicated polynomial such as $q^8-3q^6+q^5-5q^4$,
Lemma~\ref{L:positive} shows $q^8-3q^6\geq0$ for $q\geq2$ and
$q^5-5q^4\geq0$ for $q\geq5$.
Thus $q^8-3q^6+q^5-5q^4\geq0$ holds for $q\geq5$. Evaluating at $q=2,3,4$
shows that $q^8-3q^6+q^5-5q^4\geq0$ holds for $q\geq2$.
The \Magma~\cite{Magma} computer program listed in~\cite[Appendix~2]{http}
uses these ideas to verify Conjecture~\ref{C} for $1\leq n\leq37$.
\end{proof}

\section{A lower bound for $\unc(n,q)$}\label{S:lowerb}

In this section we count the uncyclic matrices $X\in\M(n,q)$ with
$\type(X)=(t-\alpha)^\l$ or $\type(X)=(t-\alpha)^\l(t-\beta)^\mu$, where
$\alpha,\beta$ are distinct elements of $\F_q$,
and $\lambda, \mu$ are partitions with $|\lambda|=n$ or 
$|\lambda|+|\mu|=n$ respectively (recall the definition of $\type(X)$ 
preceding (\ref{E:cent})). 
If Conjecture~\ref{C} were correct, then it would follow from the binomial
theorem that
\[
  \unc(n,q)\leq q^{n^2-n-1}+\frac n2 q^{n^2-n-2}+\OO(q^{n^2-n-3})
\]
where the constant involved in $\OO(q^{n^2-n-3})$ is independent of $q$. 
The main result of this section is that there is a lower bound comparable
to this conjectured upper bound.
%following theorem. This will give a
%lower bound for $\unc(n,q)$ of the magnitude we expect.

\begin{theorem}\label{T:lowerb}
If $q\geq2$ and $n\geq 3$, then
\[
  q^{n^2-n-1}\left(1+\left(\frac{n-1}{2}\right)q^{-1}-q^{-3}\right)
  <\unc(n,q).
\]
\end{theorem}

The proof uses the quantity 
$\omega(n,q):=\prod_{i=1}^n (1-q^{-i})=q^{-n^2}|\GL(n,q)|$.

\begin{lemma}\label{L:omega}
If $n\geq1$, then $(1-q^{-1})^2<1-q^{-1}-q^{-2}<\omega(\infty,q)
<\omega(n,q)\leq 1-q^{-1}$.
\end{lemma}

\begin{proof}
See Lemma~3.5 and Corollary~3.6 of \cite{NP}.
\end{proof}

Let $\alpha\in\F_q$. A matrix $X\in\M(n,q)$ is {\it $\alpha$-potent}
if its characteristic polynomial is $c_X(t)=(t-\alpha)^n$.
The map  $X\rightarrow X+(\beta-\alpha) I$ is a bijection between the
subsets of $\alpha$-potent matrices and $\beta$-potent matrices in
$\M(n,q)$. In particular, the numbers of $\alpha$-potent and unipotent
matrices in $\M(n,q)$ are equal. The number of unipotent matrices
in $\M(n,q)$ (or in $\GL(n,q)$) equals $q^{n(n-1)}$ by a theorem of
Steinberg~\cite[Theorem~6.6.1]{Carter}.
Denote by $\U(n,q,\alpha)$ the \emph{set of uncyclic $\alpha$-potent matrices}
in $\M(n,q)$. Note that $X\in\U(n,q,\alpha)$ if and only if
$\type(X)=(t-\alpha)^\l$ where $\l$ has more than one part.

Let $r(n,q)$ denote the number of uncyclic matrices $X$ in $\M(n,q)$ with
$\type(X)=f^\l$ where $f$ is a monic irreducible polynomial whose
degree divides $n$. Let $r(n,q,d)$ denote the number of such
matrices $X$ where $\type(X)=f^\l$, and $f$ has degree~$d$ for a fixed divisor
$d$ of $n$.
Thus $r(n,q)=\sum_{d \mid n}r(n,q,d)$. Estimating the size
of $r(n,q,1)$, is an important step towards  estimating
$r(n,q)$, which, in turn, will help us bound $\unc(n,q)$.

\begin{lemma}\label{L:rnq1}%\label{Rdqibound}
Let $r(n,q,1)$ denote the number of uncyclic matrices in $\M(n,q)$ that
are $\alpha$-potent for some $\alpha\in\F_q$. If $n\geq1$, then
$r(n,q,1)=c_0(n,q)q^{n^2-n-1}$ where
\[
  c_0(n,q):=q^2\left(1-\prod_{i=2}^n(1-q^{-i})\right).
\]
Moreover, $1+q^{-1}-q^{-3}\leq c_0(n,q)<1+q^{-1}+q^{-2}$ for $n\geq3$,
and $\lim_{q\to\infty} c_0(n,q)=1$.
\end{lemma}

\begin{proof} Since $|\U(n,q,\alpha)|$ is independent of $\alpha\in\F_q$,
it follows that $r(n,q,1)=q|\U(n,q,1)|$. Thus it remains to count the
uncyclic unipotent matrices. A cyclic unipotent matrix belongs to a
conjugacy class with type $(t-1)^{(n)}$,
and an uncyclic unipotent matrix $X$ has $\type(X)=(t-1)^{\l}$
for some $\l\neq(n)$. By (\ref{E:cq}), the centralizer of a cyclic unipotent
matrix has order $q^n(1-q^{-1})$. It follows, using the above mentioned theorem
of Steinberg, that
\[
  |\U(n,q,1)|=q^{n(n-1)}-
     \frac{q^{n^2}\prod_{i=1}^n(1-q^{-i})}{q^n(1-q^{-1})}
    = q^{n^2-n}\left(1-\prod_{i=2}^n(1-q^{-i})\right).
\]
The cardinality of the disjoint union $\bigcup_{\alpha\in\F_q}\U(n,q,\alpha)$
is thus
\[
  r(n,q,1)= q^{n^2-n+1}\left(1-\prod_{i=2}^n(1-q^{-i})\right)
            =c_0(n,q)q^{n^2-n-1}.
\]

It remains to estimate $c_0(n,q)$. Since $c_0(n,q)$ is an increasing
function of $n$, it follows that
$1+q^{-1}-q^{-3}=c_0(3,q)\leq c_0(n,q)<c_0(\infty,q)$ for $n\geq3$.
The following calculation shows that the limit 
\begin{equation}
  c_0(\infty,q)=1+q^{-1}+q^{-2}-q^{-5}-q^{-6}-q^{-7}-q^{-8}-q^{-9}
  +q^{-13}+q^{-14}+\cdots
\end{equation}
is finite for all $q$:
\begin{align*}
  c_0(\infty,q)&=q^2\left[1-(1-q^{-2})\prod_{i\geq3}(1-q^{-i})\right]
  <q^2\left[1-(1-q^{-2})\left(1-\sum_{i\geq3}q^{-i}\right)\right]\\
  &=q^2\left[1-(1-q^{-2})\left(1-\frac{q^{-3}}{1-q^{-1}}\right)\right]=1+q^{-1}+q^{-2}.
\end{align*}
Finally, $\lim_{q\to\infty}\left(1+q^{-1}-q^{-3}\right)=
\lim_{q\to\infty} \left(1+q^{-1}+q^{-2}\right) =1$
so $\lim_{q\to\infty} c_0(n,q)= 1$.
\end{proof}

\begin{proof}[Proof of Theorem~\ref{T:lowerb}]
By Lemma~\ref{L:rnq1} the number $r(n,q,1)$ of uncyclic matrices in $\M(n,q)$
with type $(t-\alpha)^\l$, for some $\alpha\in\F_q$ and $\l\neq(n)$,
is $q^{n^2-n-1}+q^{n^2-n-2}+\OO(q^{n^2-n-3})$.
We shall now show that the number of uncyclic matrices in $\M(n,q)$ with
type $(t-\alpha)^\l(t-\beta)^\mu$ where $\alpha\neq\beta$
is $\left(\frac {n-3}2\right) q^{n^2-n-2}+\OO(q^{n^2-n-3})$. These two
contributions give a lower bound for $\unc(n,q)$ approximately of the size
forecast in the preamble to this section.

It is easy to check using the values for $\unc(n,q)$ in Table~\ref{TableUnc}
that Theorem~\ref{T:lowerb} is true for $n=3,4$.
Assume henceforth that $n\geq5$.
We count the number of matrices $X\in\M(n,q)$ with
$\type(X)=(t-\alpha)^{(\lambda_1,1)}(t-\beta)^{(\mu_1,1)}$, for fixed elements
$\alpha\neq\beta$ in $\F_q$ and $\lambda_1\geq\mu_1\geq1$ such that
$n=\lambda_1+\mu_1+2$. It follows from Table~\ref{Tableclq} that
\begin{equation*}
  c((\lambda_1,1),q)=\left\lbrace
    \begin{array}{ll}
      q^{\lambda_1+3}(1-q^{-1})^2&\mbox{if $\lambda_1>1$,}\\
      q^{\lambda_1+3}(1-q^{-1})(1-q^{-2})\quad&\mbox{if $\lambda_1=1$.}\\
    \end{array}\right.
\end{equation*}
Since $\l_1+3+\mu_1+3=n+4$, it follows from~(\ref{E:Orbit}) that $X$ lies
in a $\GL(n,q)$-orbit of size
\[
  \frac{q^{n^2-n-4}\omega(n,q)}{(1-q^{-1})^4},\qquad
  \frac{q^{n^2-n-4}\omega(n,q)}{(1-q^{-1})^3(1-q^{-2})},\quad\text{or}\quad
  \frac{q^{n^2-n-4}\omega(n,q)}{(1-q^{-1})^2(1-q^{-2})^2}
\]
if $\l_1\geq\mu_1>1$, $\l_1>\mu_1=1$, or $\l_1=\mu_1=1$, respectively.

How many $\GL(n,q)$-orbits arise if we vary
$\alpha\neq\beta$ and $\lambda_1\geq\mu_1\geq1$?
To answer this question we consider three cases: (a) $\lambda_1>\mu_1>1$,
(b) $\lambda_1=\mu_1>1$, and (c) $\lambda_1>\mu_1=1$. (As $n\geq5$, the case
$\lambda_1=\mu_1=1$ does not arise.)
(a) If $\lambda_1>\mu_1>1$, then $n-2-\mu_1>\mu_1$ and so the values for $\mu_1$
are $2,3,\dots,\lceil\frac{n-2}{2}\rceil-1$. Thus there are $q(q-1)$
choices for $(\alpha,\beta)$ and
$\lceil\frac{n-2}{2}\rceil-2=\lceil\frac{n-6}{2}\rceil$ choices for
$\mu_1$ giving $q(q-1)\lceil\frac{n-6}{2}\rceil$ orbits.
(b) If $\lambda_1=\mu_1>1$, then $n$ is even, $\l_1=\mu_1=\frac{n-2}2$, and there
are $q(q-1)/2$
orbits as swapping $\alpha$ and $\beta$ gives a matrix in the same orbit.
(c) If $\lambda_1>\mu_1=1$, then $\l_1=n-3$ and there are $q(q-1)$ orbits.
The number of orbits in cases~(a) and~(b) combined equals
$q(q-1)\left(\frac{n-5}{2}\right)$ because if $n$ is odd 
then $\lceil\frac{n-6}{2}\rceil=\frac{n-5}{2}$, 
while if $n$ is even then $\lceil\frac{n-6}{2}\rceil+\frac{1}{2}$ also
equals $\frac{n-5}{2}$.
Thus the total number of matrices $X$ in these three cases is:
\begin{align*}
  &q(q-1)\left(\frac{n-5}{2}\right)
  \frac{q^{n^2-n-4}\omega(n,q)}{(1-q^{-1})^4}
  +q(q-1)\frac{q^{n^2-n-4}\omega(n,q)}{(1-q^{-1})^3(1-q^{-2})}\\
  &=\frac{q^{n^2-n-2}\omega(n,q)}{(1-q^{-1})^3}
  \left[\frac{n-5}{2}+\frac{1}{1+q^{-1}}\right].
\end{align*}
By Lemma~\ref{L:omega}, $\omega(n,q)>(1-q^{-1})^2$, and also
$\frac{1}{1-q^{-1}}>1+q^{-1}$. As $n\geq5$ the above expression is greater than
\begin{equation}\label{E:partofbound}
  q^{n^2-n-2}(1+q^{-1})\left[\frac{n-5}{2}+\frac{1}{1+q^{-1}}\right]
  > q^{n^2-n-2}\left(\frac{n-3}{2}\right).
\end{equation}

The number of uncyclic matrices of type $(t-\alpha)^\l$ for some $\alpha$
is by Lemma~\ref{L:rnq1} at least
\begin{equation}\label{E:alpha}
  q^{n^2-n-1}\left(1+q^{-1}-q^{-3}\right).
\end{equation}
Adding the lower bound~(\ref{E:alpha}) to the lower bound (\ref{E:partofbound})
for the number of uncyclic matrices of type
$(t-\alpha)^{(\lambda_1,1)}(t-\beta)^{(\mu_1,1)}$ gives the lower bound
\[
  \unc(n,q)>
  q^{n^2-n-1}\left(1+\left(\frac{n-1}{2}\right)q^{-1}-q^{-3}\right)
\]
of Theorem~\ref{T:lowerb}.
\end{proof}

\section{An upper bound for $\unc(n,q)$ where $q>2$}\label{S:upperb}

It surprised the authors that mathematical induction, as employed in the proof
of Theorem~\ref{T:upperb} below, could be used successfully to
find an upper bound for $\unc(n,q)$ of the form postulated in
Conjecture~\ref{C}.

First we consider uncyclic matrices involving a unique irreducible $f$. 
Let $\Irr(r,q)$ denote the set of monic degree-$r$ irreducible polynomials over
$\F_q$. Recall that $N(r,q):=|\Irr(r,q)|$, and that
$\omega(n,q):=\prod_{i=1}^n (1-q^{-i})=q^{-n^2}|\GL(n,q)|$.

%Set $N(d,q):=|\Irr(d,q)|$ where $\Irr(d,q)$ denotes the set of monic
%irreducible polynomials over $\F_q$ of degree $d$.

\begin{lemma}\label{L:rnq}
Let $r(n,q)$ denote the cardinality of the set
\[
  \{X\in\M(n,q)\mid X\textup{ is uncyclic, and }
  c_X(t)=f^{n/d}\textup{ for some $d|n$, and some $f\in\Irr(d,q)$}\}
\]
and set $c_1(n,q):=r(n,q)/q^{n^2-n-1}$. If $n\geq2$, then
$c_1(n,q)<c^*(q)$ where
\[
  c^*(q):=c_0(\infty,q)+q\frac{\omega(4,q)c_0(\infty,q^2)}{\omega(\infty,q^2)}
  \left(q\log(1-q^{-2})-\log(1-q^{-1})\right).
\]
Moreover, $1+q^{-1}-q^{-3}<c^*(q)<1+\frac{3}{2}q^{-1}+\frac{2}{3}q^{-2}$
and $\lim_{q\to\infty}c^*(q)=1$.
\end{lemma}

\begin{proof}
It follows from the remarks preceding Lemma~\ref{L:rnq1} that
\begin{equation}\label{E:sum}
  r(n,q)=r(n,q,1)+\sum_{\properdivisors} r(n,q,d)
\end{equation}
because $r(n,q,n)=0$.
Thus $r(n,q,1)\leq r(n,q)$ and so, by Lemma~\ref{L:rnq1}, 
$c_0(n,q)\leq c_1(n,q)$ with equality
if and only if $n$ is prime. It follows from Lemma~\ref{L:rnq1} that
$1+q^{-1}-q^{-3}=c_0(3,q)$ $< c_0(\infty,q)<c^*(q)$.
It remains to prove that $c_1(n,q)<c^*(q)$ for $n\geq2$ and that
$c^*(q)<$ $1+\frac{3}{2}q^{-1}+\frac{2}{3}q^{-2}$. The first inequality is
true when $n=2,3$ by Lemma~\ref{L:rnq1} as
\[
  c_1(2,q)=c_0(2,q)=1<c_1(3,q)=c_0(3,q)=1+q^{-1}-q^{-3}
  < c_0(\infty,q)<c^*(q).
\]
Assume henceforth that $n\geq4$.

We digress to generalize the formula for $r(n,q,1)=c_0(n,q)q^{n^2-n-1}$
in Lemma~\ref{L:rnq1} to $r(n,q,d)$.
It follows from (\ref{E:cq}) and (\ref{E:Orbit}) that
\[
  r(n,q,1)=N(1,q)\sum_{\genfrac{}{}{0pt}{}{\lambda\vdash n}{\lambda\neq(n)}}
             \frac{|\GL(n,q)|}{c(\lambda,q)}\quad\text{and}\quad
  r(n,q,d)=
  N(d,q)\sum_{\genfrac{}{}{0pt}{}{\lambda\vdash \frac{n}{d}}{\lambda\neq(\frac{n}{d})}}
             \frac{|\GL(n,q)|}{c(\lambda,q^d)}
\]
where the sums are over all partitions with more than one
part. Note that the first sum counts the elements of the disjoint union
$\bigcup_{\alpha\in\F_q}\U(n,q,\alpha)$.
Relating these formulas gives
\begin{align*}
  r(n,q,d)&=\frac{N(d,q)|\GL(n,q)|}{N(1,q^d)|\GL(\frac nd,q^d)|}
               \,N(1,q^d)\sum_{\genfrac{}{}{0pt}{}{\lambda\vdash \frac{n}{d}}{\lambda\neq(\frac{n}{d})}}
               \frac{|\GL(\frac {n}{d},q^d)|}{c(\lambda,q^d)}\\
            &=\frac{N(d,q)|\GL(n,q)|}{N(1,q^d)|\GL(\frac{n}{d},q^d)|}
               r({\frac{n}{d}},q^d,1).
\end{align*}

By Lemma~\ref{L:rnq1} we have $r(n,q,1)=c_0(n,q)q^{n^2-n-1}$, and so
\begin{equation}\label{E:rnqd}
  \begin{aligned}
    r(n,q,d)&=
    \frac{q^{-d}N(d,q)q^{n^2}\omega(n,q)}{(q^d)^{(n/d)^2}\omega(\frac{n}{d},q^d)}
                 c_0(\frac{n}{d},q^d)(q^d)^{(n/d)^2-(n/d)-1}\\
            &=\frac{\omega(n,q)c_0(\frac{n}{d},q^d)}{\omega(\frac{n}{d},q^d)}
                \left[q^{-d}N(d,q)\right]q^{n^2-n-d}
  \end{aligned}
\end{equation}
Since $n\geq4$ and $1<d<n$, each of $d$ and $n/d$ is at least 2, and so we have
\begin{equation}\label{E:gamma}
  \frac{\omega(n,q)c_0({\frac{n}{d}},q^d)}{\omega(\frac{n}{d},q^d)}
  <\frac{\omega(4,q)c_0(\infty,q^2)}{\omega(\infty,q^2)}=:\gamma(q).
\end{equation}
It follows from (\ref{E:sum}),(\ref{E:rnqd}) and (\ref{E:gamma}) that
\[
  r(n,q)=r(n,q,1)+\sum_{\properdivisors} r(n,q,d)
    \leq c_0(\infty,q)q^{n^2-n-1}
     +\gamma(q)\left(\sum_{\properdivisors} q^{-d}N(d,q) q^{n^2-n-d}\right).
\]
This proves that $r(n,q)\leq K(n,q)q^{n^2-n-1}$ where
\[
  K(n,q):=c_0(\infty,q)+q\gamma(q)\sum_{\properdivisors} q^{-2d}N(d,q).
\]
Thus $c_1(n,q)\leq K(n,q)$, and our goal now is to prove that
$K(n,q)<c^*(q)$ for $n\geq4$.

The bound $N(d,q)\leq(q^d-q)/d$, which holds for $d\geq2$, gives
\begin{equation}\label{E:estimate}
  \sum_{\properdivisors} q^{-2d}N(d,q)
  \leq\sum_{d\geq2} \frac{q^{-d}}{d}-q\sum_{d\geq2}\frac{q^{-2d}}{d}
  =\sum_{d\geq1} \frac{q^{-d}}{d}-q\sum_{d\geq1}\frac{q^{-2d}}{d}.
\end{equation}
The series $\sum_{d\geq1}\frac{x^{d}}{d}$ converges absolutely for $|x|<1$
to $-\log(1-x)$. Thus
\[
  K(n,q)<c_0(\infty,q)+q\gamma(q)\left(q\log(1-q^{-2})-\log(1-q^{-1})\right)
  =c^*(q),
\]
so $c_1(n,q)\leq K(n,q)<c^*(q)$ for $n\geq4$. Finally we must show that
$c^*(q)<1+\frac{3}{2}q^{-1}+\frac{2}{3}q^{-2}$.

We begin by showing
$q\log(1-q^{-2})-\log(1-q^{-1})<q^{-2}/2$ for $q\geq2$.
This is true when $q=2$ because $2\log(3/4)-\log(1/2)<0.125$.
Suppose now that $q\geq3$. If $0\leq x<1$, then elementary calculus gives
\[
  x+\frac{x^2}2\leq-\log(1-x)
  \leq x+\frac{x^2}2+\sum_{d\geq3}\frac{x^d}3=x+\frac{x^2}2+\frac{x^3}{3(1-x)}.
\]
If $x=q^{-2}$, then $q^{-2}+\frac{q^{-4}}{2}\leq-\log(1-q^{-2})$ and
$q\log(1-q^{-2})\leq -q^{-1}-\frac{q^{-3}}{2}$.
If $x=q^{-1}$, then $-\log(1-q^{-1})<q^{-1}+\frac{q^{-2}}2+\frac{q^{-3}}{3(1-q^{-1})}\leq
q^{-1}+\frac{q^{-2}}2+\frac{q^{-3}}{2}$ for $q\geq3$. Adding shows
\begin{equation}\label{E:loglog}
  q\log(1-q^{-2})-\log(1-q^{-1})<\frac{q^{-2}}{2}\qquad\text{for $q\geq3$}.
\end{equation}

By Lemma~\ref{L:sharp} below, $c^*(2)<1.83<\frac{23}{12}$, and hence
the bound $c^*(q)<1+\frac{3}{2}q^{-1}+\frac{2}{3}q^{-2}$ holds when $q=2$. 
Assume henceforth that $q\geq3$.
Lemma~\ref{L:rnq1} gives $c_0(\infty,q)\leq 1+q^{-1}+q^{-2}$,
and hence $c_0(\infty,q^2)\leq 1+q^{-2}+q^{-4}$.
Lemma~\ref{L:omega} implies $\omega(\infty,q^2)>1-q^{-2}-q^{-4}$, and 
Lemma~\ref{L:positive} may be used to show that
$\omega(\infty,q^2)^{-1}<1+q^{-2}+3q^{-4}$ for $q\geq3$. The inequalities
$\omega(4,q)<(1-q^{-1})(1-q^{-2})$ and (\ref{E:loglog}) give:
\def\+{\kern-1.5pt+\kern-1.5pt}
\def\-{\kern-1.5pt-\kern-1.5pt}
\begin{align*}
  c^*(q)&<(1+q^{-1}+q^{-2})+(1-q^{-1})(1-q^{-2})(1+q^{-2}+q^{-4})(1+q^{-2}+3q^{-4})
  \frac{q^{-1}}{2}\\
  &=1+\frac{1}{2}\left(3q^{-1}\+q^{-2}\+q^{-3}\-q^{-4}\+3q^{-5}\-3q^{-6}\-q^{-7}\+q^{-8}\-q^{-9}\+q^{-10}\-3q^{-11}\+3q^{-12}\right)\\
   &\leq 1+\frac{3}{2}q^{-1}+\frac{2}{3}q^{-2},
\end{align*}
where the final inequality follows from Lemma~\ref{L:positive} with $q_0=3$.
As $q$ approaches infinity, the established lower and upper bounds
for $c^*(q)$ both approach~1. Thus $\lim_{q\to\infty} c^*(q)=1$ as claimed.
This completes the proof.
\end{proof}

The proof of our main theorem requires sharper bounds for $c^*(2)$ and $c^*(3)$
than those provided by Lemma~\ref{L:rnq}.

\begin{lemma}\label{L:sharp}
For $m\geq2$, $q\geq2$, we have 
\begin{equation}\label{E:OmegaEst}
  \omega(\infty,q)>\omega(m-1,q)\left(1-\frac{q^{-m}}{1-q^{-1}}\right)
\end{equation}
and this bound may be used to show that
$c^*(2)<1.83$ and $c^*(3)<1.56$.
\end{lemma}

\begin{proof}
The bound $\prod_{i=m}^\infty(1-q^{-i})>1-\sum_{i=m}^\infty q^{-i}$ gives rise
to the lower bound (\ref{E:OmegaEst}) for $\omega(\infty,q)$.
This, in turn, gives an upper bound for $c_0(\infty,q)$
(see Lemma~\ref{L:rnq1} for a definition).
Setting $m=6$ and $q=2,4$ in (\ref{E:OmegaEst}) gives
\[
  \omega(\infty,2)>0.28869,\quad\omega(\infty,4)>0.688,\quad c_0(\infty,2)<1.691,
  \quad\text{and}\quad c_0(\infty,4)<1.312.
\]
Similarly, setting $m=4$ and $q=3,9$ in (\ref{E:OmegaEst}) gives
\[
  \omega(\infty,3)>0.560,\quad\omega(\infty,9)>0.876,\quad c_0(\infty,3)<1.439,
  \quad\text{and}\quad c_0(\infty,9)<1.124.
\]
These inequalities give $c^*(2)<1.83$ and $c^*(3)<1.56$.
\end{proof}
%%%%%%%%%%%%%%%%%%%%%%%%%%%%%%%%%%%%%%%%%%%%%%%%%%%%%

\begin{theorem}\label{T:upperb}
If $n\geq1$, then $\unc(n,3)<(1.56)3^{n^2-n-1}\left(1.59\right)^n$ and
\begin{equation}\label{E:ub}
  \unc(n,q)< c^*(q) q^{n^2-n-1}\left(1+q^{-1}+2q^{-2}\right)^n
\end{equation}
for $q\geq4$, where $c^*(q)$ is defined in \textup{Lemma~\ref{L:rnq}}, and
satisfies $1<c^*(q)<1.56$ for $q\geq3$.
\end{theorem}

\begin{proof}
Our proof has two parts. First, we use induction on $n$ and
a geometric argument to prove $\unc(n,q)\leq c^*(q)q^{n^2-n-1}\rho(q)^n$ for
$n\geq1$ and $q\geq2$, where
\begin{equation}\label{E:ind}
  \rho(q):=\frac{1+\sqrt{1+\frac{4c^*(q)}{q\omega(\infty,q)}}}{2}.
\end{equation}
Second, we prove that $\rho(3)<1.59$, and
$\rho(q)<1+q^{-1}+2q^{-2}$ for $q\geq4$.

It follows from the definition (\ref{E:ind}) that $\rho(q)>1$ for all $q\geq2$.
A simple calculation shows that $\unc(n,q)\leq c^*(q)q^{n^2-n-1}\rho(q)^n$
is true for $n=1,2$ and all $q$. Consider the proof for $n=3$.
By Table~\ref{TableUnc}, $\unc(3,q)=q^5+q^4-q^2$ and so
the inequality to be proved is:
\[
q^5\left(1+q^{-1}-q^{-3}\right)\leq c^*(q)q^5\rho(q)^3. 
\]
Now by Lemma~\ref{L:rnq}, $1+q^{-1}-q^{-3}<c^*(q)$, 
and as $\rho(q)>1$ the inequality above holds for all $q$.
%The binomial theorem gives:
%\[
%  \rho(q)>\frac{1+\sqrt{1+\frac{4}{q(1-q^{-1})}}}{2}
%  >\frac{1+\left(1+\frac{2}{q(1-q^{-1})}\right)}{2}=\frac{1}{1-q^{-1}}
%\]
%for $q>5$. 
%By Table~\ref{TableUnc}, $\unc(3,q)=q^5+q^4-q^2$ and so
%\[
%  \unc(3,q)=q^5\left(1+q^{-1}-q^{-2}\right)<\frac{q^5}{1-q^{-1}}
%  <q^5\rho(q)<c^*(q)q^5\rho(q)^5=c^*(q)q^{3^2-3-1}\rho(q)^5.
%\]
%The cases $n=3$ and $q\leq5$ are readily checked by hand.
Assume henceforth that $n\geq4$.

By definition, there are precisely $r(n,q)$ uncyclic matrices
$X\in\M(n,q)$ for which $c_X(t)$ is a power of some irreducible polynomial.
We shall now over-estimate the number of uncyclic $X$ for which $c_X(t)$
is {\it not} a power of a single irreducible.

We impose an arbitrary total ordering on the (finite number of) irreducible 
polynomials over $\F_q$ of degree at most $n$.
For each uncyclic matrix $X$ such that $c_X(t)$
is not a power of an irreducible, there exists at least one irreducible polynomial $f$ such that, 
if $f^{\nu(f)}$ is the highest power of $f$ dividing $c_X(t)$, 
then $0<d(f)\nu(f)\leq n/2$.
We choose the first irreducible $f$ in the total ordering with this property.
Write $V=U\oplus W$, where $U=\ker f(X)^{\nu(f)}$ is the $f$-primary component
and $W=\im\, f(X)^{\nu(f)}$ is an $X$-invariant complement. The restrictions
$X_U$ and $X_W$ of $X$ to $U$ and $W$ are both uncyclic. Moreover, $X$
determines a unique 4-tuple $(U,W,X_U,X_W)$. Counting the number of possible
4-tuples will give an upper bound for the number of $X$.

Set $k:=\dim(U)$. Then $k=d(f)\nu(f)\leq n/2$, and $k\geq 2$ as $X_U$
is uncyclic. The number of decompositions $V=U\oplus W$ with
$\dim(U)=k$ is
\[
  \frac{|\GL(n,q)|}{|\GL(k,q)|\,|\GL(n-k,q)|}.
\]
The number of choices for $X_U$ is precisely $r(k,q)$, and the number
of choices for $X_W$ is at most $\unc(n-k,q)$. (At this point the
reader may be concerned that we are not using the fact that the characteristic
polynomial of $X_W$ is coprime to $f$. It is remarkable that this otherwise
very delicate counting problem is essentially insensitive to such an
over-estimation.) Thus
\[
  \unc(n,q)\leq r(n,q)+\sum_{k=2}^{\lfloor\frac {n}{2}\rfloor}
            \frac{|\GL(n,q)|}{|\GL(k,q)|\,|\GL(n-k,q)|}\, r(k,q)\,\unc(n-k,q).
\]

We shall abbreviate $\rho(q)$, $c^*(q)$ and $\omega(\infty,q)$ by $\rho$,
$c^*$ and $\omega$, respectively.
As $n-k<n$, it follows by induction that
\[
  \unc(n-k,q)\leq c^*q^{(n-k)^2-(n-k)-1}\rho^{n-k}.
\]
Moreover, Lemma~\ref{L:rnq} gives
$r(k,q)=c_1(k,q)q^{k^2-k-1}\leq c^* q^{k^2-k-1}$ for all $k$ and, since
$\frac{\omega(n,q)}{\omega(n-k,q)}=\prod_{i=n-k+1}^n(1-q^{-i})<1$, we have
%by Lemma~\ref{L:positive}
\[
  \frac{|\GL(n,q)|}{|\GL(k,q)|\,|\GL(n-k,q)|}=
  \frac{\omega(n,q)q^{n^2-k^2-(n-k)^2}}{\omega(k,q)\omega(n-k,q)}
  < \frac {q^{n^2-k^2-(n-k)^2}}{\omega(k,q)}.
\]
Thus
\begin{equation*}
  \unc(n,q)\leq c^*q^{n^2-n-1}+\sum_{k=2}^{\lfloor\frac {n}{2}\rfloor}
            \frac {q^{n^2-k^2-(n-k)^2}}{\omega(k,q)}
            c_1(k,q)q^{k^2-k-1}\,c^*q^{(n-k)^2-(n-k)-1}\rho^{n-k}.
\end{equation*}
The exponent of $q$ in the terms of the summation is independent of $k$ as
\[
  n^2-k^2-(n-k)^2+k^2-k-1+(n-k)^2-(n-k)-1=n^2-n-2.
\]
Therefore
\begin{equation}\label{estimate}
  \unc(n,q)\leq c^*q^{n^2-n-1}\left(1+
  \sum_{k=2}^{\lfloor\frac n2\rfloor} \frac{c_1(k,q)}{q\omega(k,q)}\rho^{n-k}
  \right).
\end{equation}
To complete the induction we must show that the above bracketed expression is
at most~$\rho^n$.
Towards this end, note that $\frac{c_1(k,q)}{\omega(k,q)} <\frac{c^*}{\omega}$
by Lemma~\ref{L:rnq}.
Since $\lfloor \frac n2\rfloor+\lceil \frac n2\rceil=n$, $n\geq4$ and
$\rho>1$, we have
\[
  \sum_{k=2}^{\lfloor\frac n2\rfloor} \rho^{n-k}=
  \rho^{\lceil \frac n2\rceil}+\rho^{\lceil \frac n2\rceil+1}
      +\cdots+\rho^{n-2}
    =\frac{\rho^{\lceil \frac n2\rceil}(\rho^{\lfloor\frac n2\rfloor-1}-1)}%
          {\rho-1}\leq\frac{\rho^{n-1}-\rho^2}{\rho-1}.
\]

It follows from the definition (\ref{E:ind}) of $\rho$, by rationalizing
the denominator, that
\[
  \frac{1}{\rho-1}=\frac{2}{-1+\sqrt{1+\frac{4c^*}{q\omega}}}=
  \frac{q\omega}{2c^*}
  \left[1+\sqrt{1+\frac{4c^*}{q\omega}}\;\right]=\frac{q\omega\rho}{c^*}.
\]
The previous three displayed equations now give
\[
  \unc(n,q)\leq c^*q^{n^2-n-1}\left(1+\frac {c^*}{q\omega}
  \frac{q\omega\rho}{c^*} (\rho^{n-1}-\rho^2)\right)=
   c^*q^{n^2-n-1}\left( 1+\rho^n-\rho^3\right).
\]
Since $\rho>1$, it follows that $1-\rho^3<0$. Thus
$\unc(n,q)< c^*q^{n^2-n-1}\rho^n$ and we have completed the inductive proof.

To complete the proof, we must estimate $\rho(q)$. By Lemma~\ref{L:sharp},
$c^*(3)<1.56$ and $\omega(\infty,3)<0.56$. Thus $\rho(3)<1.59$,
and the inequality for $\unc(n,3)$ follows.
Assume now that $q\geq4$. We will show that
\begin{equation}\label{E:rho1}
  \rho(q)=\frac{1+\sqrt{1+\frac{4c^*}{q\omega}}}{2}<1+q^{-1}+2q^{-2}.
\end{equation}
Multiplying (\ref{E:rho1})
by~2, subtracting~1, and squaring gives that (\ref{E:rho1}) is equivalent to
\begin{equation}\label{E:rho2}
  1+\frac{4c^*}{q\omega}<(1+2q^{-1}+4q^{-2})^2
  =1+4q^{-1}+12q^{-2}+16q^{-3}+16q^{-4}.
\end{equation}
Subtracting~1 from (\ref{E:rho2}) and multiplying by the positive quantity
$\frac{q\omega}{4}$ gives the equivalent inequality
\begin{equation}\label{E:rho3}
  c^*<\omega(1+3q^{-1}+4q^{-2}+4q^{-3}).
\end{equation}
By virtue of the inequalities $c^*<1+\frac{3}{2}q^{-1}+\frac{2}{3}q^{-2}$ 
from Lemma~\ref{L:rnq}, and $1-q^{-1}-q^{-2}<\omega$ from
Lemma~\ref{L:omega}, the inequality (\ref{E:rho3}), and hence 
also the required equivalent inequality (\ref{E:rho1}), will follow from a proof of 
the following stronger inequality:
\begin{equation}\label{E:rho4}
  1+\frac{3}{2}q^{-1}+\frac{2}{3}q^{-2}
  <(1-q^{-1}-q^{-2})(1+3q^{-1}+4q^{-2}+4q^{-3}).
\end{equation}
Expanding and rearranging~(\ref{E:rho4}) gives
\begin{equation}\label{E:rho5}
  0< \frac{q^{-1}}{2}-\frac{2q^{-2}}{3}-3q^{-3}-8q^{-4}-4q^{-5}.
\end{equation}
This inequality is true for $q\geq4$ by Lemma~\ref{L:positive}
with $q_0=4$. 
%This completes the proof, as the steps from (\ref{E:rho1}) to (\ref{E:rho5}) may be reversed.
Thus (\ref{E:rho1}) holds for $q\geq4$. This completes the proof.
\end{proof}

\begin{corollary}
If $n\geq1$ and $q\geq3$, then the probability $p$ that a uniformly 
distributed random
$n\times n$ matrix over $\F_q$ is $f$-cyclic satisfies
\begin{equation}\label{E:estimate2}
  1-k(q)q^{-1}\left(q^{-1}+q^{-2}+\frac{12}{5}q^{-3}\right)^n<p\leq1
\end{equation}
where $k(q)=1+\frac{3}{2}q^{-1}+\frac{2}{3}q^{-2}$.
\end{corollary}

\begin{proof}
Note that $p=1-\unc(n,q)q^{-n^2}\leq1$. Theorem~\ref{T:upperb} with $q=3$ gives
\[
  \frac{\unc(n,3)}{3^{n^2}}<\frac{1.56}{3}\left(\frac{1.59}{3}\right)^n
  <\frac{k(3)}{3}\left(3^{-1}+3^{-2}+\frac{12\cdot3^{-3}}{5}\right)^n.
\]
Thus the lower bound for $p$ in (\ref{E:estimate2}) holds for $q=3$.
Assume now that $q\geq4$. Since $c^*(q)<k(q)$ by Lemma~\ref{L:rnq},
it follows from Theorem~\ref{T:upperb} that
\[
  \frac{\unc(n,q)}{q^{n^2}}<k(q)q^{-1}\left(q^{-1}+q^{-2}+2q^{-3}\right)^n
  <k(q)q^{-1}\left(q^{-1}+q^{-2}+\frac{12}{5}q^{-3}\right)^n.
\]
This establishes the lower bound for $p$ in (\ref{E:estimate2}) for $q\geq4$,
and completes the proof.
\end{proof}

\section{An upper bound for $\unc(n,2)$}\label{S:2upperb}

Theorem~\ref{T:upperb} shows that $\unc(n,q)/q^{n^2}=\OO(R(q)^n)$, where
$R(q) = \rho(q)/q$ with $\rho(q)$ as defined in (\ref{E:ind}). 
For this value of $\rho(q)$, the proof of Theorem~\ref{T:upperb} yields an upper bound
for $\rho(q)$, and hence also for $R(q)=\rho(q)/q$, as listed in
Table~\ref{Tablerho}, for various values of $q$. (The values of 
these bounds have been rounded up to the
nearest $10^{-2}$.)
We note that the inductive part of the proof of 
Theorem~\ref{T:upperb} is valid
for $q=2$, but it gives an upper bound for $R(2)$ greater than $1$, 
or equivalently for $\rho(2)$ greater than $2$.
Stronger arguments are needed to show that $\unc(n,2)/2^{n^2}=\OO(R(2)^n)$
with $R(2)<1$.
If Conjecture~\ref{C} were true, then this would hold with
$R(2)\leq \frac{1}{2}+\frac{1}{2\times 2^2}=0.625$ (and 
hence with $\rho(2)=2R(2)\leq1.25$).
In this section we modify the proof of Theorem~\ref{T:upperb} to
obtain a value of $R(2)$ less than $0.983$, or $\rho(2)$ less than $1.966$,
which is still substantially larger than the bound conjectured to hold
in Conjecture~\ref{C}. Theorem~\ref{T:q=2} below implies Theorem~\ref{T:simple}.
\begin{table}[!ht]
  \begin{center}
  \begin{tabular}{|l|c|c|c|c|c|c|c|c|c|c|c|c|c|} \hline
    $q$&2&3&4&5&7&8&9&11&13&16&17&19&23\\ \hline
    $R(q)\leq$&1.18&0.53&0.35&0.26&0.15&0.17&0.13&0.11&0.09&0.07&0.07&0.06&0.05\\ \hline
    $\rho(q)\leq$&2.35&1.59&1.38&1.28&1.18&1.16&1.14&1.11&1.09&1.07&1.07&1.06&1.05\\ \hline
  \end{tabular}
\vskip2mm
  \caption{Upper bounds for $R(q)=\rho(q)/q$ and $\rho(q)$ obtained in Theorem~\ref{T:upperb}.}\label{Tablerho}
  \end{center}
\end{table}

\begin{theorem}\label{T:q=2}
If $n\geq1$, then
\quad$2^{n^2-n-1}(\frac{n}{4}+\frac{5}{8})<\unc(n,2)<(1.83)2^{n^2-n-1}(1.966)^n$.
\end{theorem}

\begin{proof} The lower bound follows from Theorem~\ref{T:lowerb}.
The upper bound is proved by adapting the inductive proof
of Theorem~\ref{T:upperb}.
By Proposition~\ref{P:conjbound} we know that $\unc(n,2)$ is at most
$2^{n^2-n-1}(1.25)^n$
for $n\leq9$ (indeed even for $n\leq37$), so the weaker bound
$\unc(n,2)<(1.83)2^{n^2-n-1}(1.966)^n$ certainly holds for $n\leq9$.
Assume henceforth that $n\geq10$. Lemma~\ref{L:sharp} shows that $c^*(2)$
as defined in Lemma~\ref{L:rnq} satisfies $c^*(2)<1.83$. Set $\rho:=1.966$.
The first part of the proof of Theorem~\ref{T:upperb} is valid for $q=2$,
and in particular,
the inequality (\ref{estimate}) holds for $q=2$. To complete the 
inductive step in the proof it is sufficient to prove, for $n\geq10$, that
\[
  1+\sum_{k=2}^{\lfloor\frac{n}{2}\rfloor}
    \frac{c_1(k,2)}{2\omega(k,2)}\rho^{n-k}\leq\rho^n
  \qquad\text{or equivalently,}\qquad
  \rho^{-n}+\sum_{k=2}^{\lfloor\frac{n}{2}\rfloor}
    \frac{c_1(k,2)}{2\omega(k,2)}\rho^{-k}\leq1
\]
with $c_1(k,2)$ as defined in Lemma~\ref{L:rnq}.
Since $\rho^{-n}\leq\rho^{-10}$ it is sufficient to prove that
\begin{equation}\label{E:one}
  \rho^{-10}+\sum_{k=2}^\infty\frac{c_1(k,2)}{2\omega(k,2)}\rho^{-k}\leq1.
\end{equation}
For $k\geq6$ we use the bounds from Lemmas~\ref{L:rnq} and~\ref{L:sharp} to obtain
%\begin{equation}\label{E:approx}
%  \frac{c_1(k,2)}{2\omega(k,2)}<\frac{c_1(\infty,2)}{2\omega(\infty,2)}
%    <\frac{c^*(2)}{2\omega(\infty,2)}<3.17,
%\end{equation}
\begin{equation}\label{E:approx}
  \frac{c_1(k,2)}{2\omega(k,2)}
    <\frac{c^*(2)}{2\omega(\infty,2)}< \frac{1.83}{2\times 0.28869}<3.17.
\end{equation}
For $k\leq5$ we use the exact values of $\frac{c_1(k,2)}{2\omega(k,2)}$.
Recall from the definitions of $c_0(k,q)$ and $c_1(k,q)$ in Lemmas~\ref{L:rnq1}
and~\ref{L:rnq} that $c_1(k,q)$ 
equals $c_0(k,q)$ when $k$ is prime.
Hence $c_1(k,2)$ equals $1,\frac{11}{8},\frac{6619}{4098}$ for $k=2,3,5$.
To compute $c_1(4,2)$, we use the proof of Lemma~\ref{L:rnq} to show
$r(4,2)=r(4,2,1)+r(4,2,2)=3152+112=3264$.
Thus $c_1(4,2)=r(4,2)/2^{11}=\frac{51}{32}$.
%We can now list the exact values of $\frac{c_1(k,2)}{2\omega(k,2)}$
%for $k\leq5$:
\vskip-1mm
\renewcommand{\arraystretch}{1.2}    % so lines in tables are not crowded
\def\ds{\displaystyle}\def\phan{\phantom{I^{I^I}}}
\begin{table}[ht!]
  \begin{center}
  \begin{tabular}{|l|c|c|c|c|} \hline
    $k$&2&3&4&5\\ \hline
%  $\ds\frac{c_1(k,2)}{2\omega(k,2)}$&$\ds\frac{4}{3}$&$\ds\frac{44}{21}$
%    &$\ds\frac{272}{105}$&$\ds\frac{26476}{9765}$\\ \hline
    $\ds\frac{c_1(k,2)^{\phan}}{2\omega(k,2)}$&$\ds\frac{4}{3}$&$\ds\frac{44}{21}$
      &$\ds\frac{272}{105}$&$\ds\frac{26476}{9765}$\\ \hline
  \end{tabular}
\vskip2mm
  \caption{Values of $\frac{c_1(k,2)}{2\omega(k,2)}$ for $2\leq k\leq5$.}\label{Tableck2}
  \end{center}
\end{table}
\vskip-3mm
Using (\ref{E:approx}) and Table~\ref{Tableck2}, the infinite sum
in~(\ref{E:one}) is less than
\[
  \sum_{k=2}^5\frac{c_1(k,2)}{2\omega(k,2)}\rho^{-k}
  +\frac{c^*(2)}{2\omega(\infty,2)}\sum_{k=6}^\infty\rho^{-k}
  <\frac{4}{3}\rho^{-2}+\frac{44}{21}\rho^{-3}+\frac{272}{105}\rho^{-4}+
  \frac{26476}{9765}\rho^{-5}+\frac{3.17\rho^{-6}}{1-\rho^{-1}}.
\]%%%\frac{c_1(k,2)}{2\omega(k,2)}
Evaluating the expression
\[
  \rho^{-10}+\frac{4}{3}\rho^{-2}+\frac{44}{21}\rho^{-3}
  +\frac{272}{105}\rho^{-4}+
  \frac{26476}{9765}\rho^{-5}+\frac{3.17\rho^{-6}}{1-\rho^{-1}}
\]
at $\rho=1.966$ gives the number $0.9992\cdots<1$. % Indeed less than 0.9993.
This completes the inductive proof.
\end{proof}

\begin{comment}
Replacing 5 in above proof by 37 and using the same method we may reduce to
$\rho=1.9606$. The expression gives 0.99987655... which is less than 1.
The gain is not worth the complexity.
\end{comment}

\section{Finding a witness to $X$ being $f$-cyclic}\label{S:witness}

In this section $h$ always denotes a {\it monic} irreducible polynomial.
Henceforth we shall consistently omit the adjective ``monic''.
The $h$-primary component $V(h)$ of an $\F_q[X]$-module $V$ can be
generalized to $V(g)$ where $g$ is a (possibly reducible) divisor of $c_X(t)$:
set $V(g):=\bigoplus_{h|g} V(h)$ where the sum is over irreducible divisors
$h$ of $g$.

The Holt-Rees \MeatAxe\ algorithm~\cite[Section 2]{HR} initially finds a random matrix $X$,
and then begins by executing the following steps:\newline
\begin{tabular}{l}
(1) find an irreducible
factor $g$ of the characteristic polynomial $c_X(t)$,\\ 
(2) evaluate $g(t)$ at $X$ to compute $Y=g(X)$, and\\ 
(3) find a non-zero vector $u\in\ker(Y)$.\\ 
\end{tabular}

The matrix $X$ can be used to prove
irreducibility if it is $f$-cyclic relative to $g$, 
that is, if (and only if) the degree of $g$ equals
$\dim(\ker(Y))$. Step~(2) has cost $\OO(\Mat(n)n)$ field operations\footnote{
A lower complexity can be achieved by conjugating $X$ into Frobenius
normal form, evaluating $g(t)$ at the matrix obtained, and conjugating back.
For the complexity of this approach
see: C. Pernet and A. Storjohann, Frobenius form in expected
matrix multiplication time over sufficiently large fields, preprint.}
(that is, additions, subtractions, multiplications, and inversions in $\F_q$), 
where $\Mat(n)$ is an upper bound for the number of field operations
required to multiply two matrices in $\M(n,q)$. 
The purpose of this section is to present a one-sided
Monte Carlo algorithm called \IsfCyclic\ that requires (only)
$\OO(\Mat(n)\log n)$
field operations, and in particular obviates the necessity of applying 
the rather expensive Step (2).

Given an $f$-cyclic matrix $X\in\M(n,q)$, and a positive real number $\eps<1$,
this algorithm returns {\sc True} with probability at least $1-\eps$.
Moreover in this case it constructs a divisor $g$ of $c_X(t)$ and a non-zero
vector $u$ such that $\gcd(g,c_X/g)=1$ and $V(g)=u\F_q[X]$. This shows that
$X$ is $f$-cyclic relative to every irreducible divisor of~$g$. 
If \IsfCyclic\ fails to construct $g,u$ with these properties then it returns {\sc False},
that is to say, \IsfCyclic\ incorrectly reports `$X$ is not $f$-cyclic'. However, the probability of this happening is at most $\eps$.
On the other hand, if $X$ is not $f$-cyclic, then \IsfCyclic\ correctly
returns {\sc False}. In summary, an output {\sc True} is always correct, 
while an output {\sc False} is incorrect with probability at most $\eps$.
These assertions are proved in Theorem~\ref{T:cost}.

If it were desirable that the polynomial $g$ returned by the algorithm
\IsfCyclic\ be irreducible, then \IsfCyclic\ could be modified to
incorporate a randomised polynomial factorisation algorithm.
%The algorithm \IsfCyclic\ does not check whether $g$ is irreducible. If this 
%information were required, then \IsfCyclic\ could be modified to
%incorporate a randomised polynomial factorisation algorithm.

%actually compute $f$; that requires factoring. When we analyse
%the complexity of \IsfCyclic\ we shall assume that $|F|=q$ is finite.

\subsection{Witnesses and orders}\label{witness}

Given a matrix $X\in\M(n,q)$ and a non-constant divisor~$g$ of
$c_X(t)=\prod_ff^{\nu(f)}$, a vector $v\in V:=\F_q^{1\times n}$
is called a $g$-{\it witness} for $X$ if the cyclic submodule $v\F_q[X]$
contains the $h$-primary component $V(h)$ of $V$ for all irreducible
divisors $h$ of $g$. 
The following are equivalent:~(1) $v$ is a $g$-witness for $X$,
(2)~$V(g)\subseteq v\F_q[X]$, and (3)  $\prod_hh^{\nu(h)}$ divides the
order polynomial $\ord_X(v)$, where the product is over all irreducible 
divisors $h$ of $g$. (Recall that $a(t)=\ord_X(v)$ is the smallest degree
monic polynomial over $\F_q$ satisfying~$va(X)=0$.) 
As submodules of cyclic modules
are cyclic, \emph{$X$ has a $g$-witness $v$ if and only if $X$ is $f$-cyclic
relative to every irreducible divisor $h$ of $g$}.

It turns out that a
matrix $X$, which is $f$-cyclic relative to every irreducible divisor of $g$, 
has many $g$-witnesses,
and failure to find a $g$-witness (for any such $g$) provides
``probabilistic evidence'' that $X$ is uncyclic (as is shown below).

Recall the following notation from Section~\ref{S:conj}
\[
  \type(X)=\prod_h h^{\lambda(h)},\quad c_X(t)=\prod_h h^{|\lambda(h)|},\quad 
  m_X(t)=\prod_h h^{\lambda(h)_1},\quad V(h)=\ker h(X)^{\lambda(h)_1},
\]
and set $V(h)_k:=\ker h(X)^{\lambda(h)_1-k}$ for $0\leq k\leq\lambda(h)_1$.
The subspaces $V(h)_k$ define a chain
\begin{equation}\label{E:comp}
  V(h)=V(h)_0>V(h)_1>\cdots>V(h)_{\lambda(h)_1}=0
\end{equation}
of $\F_q[X]$-submodules. 

We introduce the notion of the {\it $h$-order} of a vector or polynomial,
see~\cite[7.17]{BAII}. Fix an irreducible polynomial $h(t)$, and let $I$ be the
ideal $h(t)\F_q[t]$ of $\F_q[t]$. A non-zero vector $v$ in an $\F_q[t]$-module $M$
is said to have $h$-order $k$, written $o_h(v)=k$, if $k$ is the largest
integer such that $v\in MI^k$. By convention we set $o_h(0):=\infty$. 
In our applications, the module $M$
will be either $V$, the $h$-primary component $V(h)$, or the ring $\F_q[t]$.
We denote elements of $V$ by $u,v$, and elements of
$\F_q[t]$ by $a,d,e,g$.
In the case when $M=\F_q[t]$, we have $\cap_{k\geq0}MI^k=0$, and
$o_h$ is an exponential valuation satisfying:
(i)~$o_h(a)=\infty$ if and only if $a=0$, 
(ii)~$o_h(ab)=o_h(a)+o_h(b)$, 
(iii)~$o_h(a+b)\geq\min\left(o_h(a),o_h(b)\right)$, and
(iv)~$o_h(\gcd(a,b))=\min\left(o_h(a),o_h(b)\right)$. When $M=V(h)$ properties
(i) and (iii) hold.

Suppose that $v\in V(h)$. Then $o_h(v)=k$ holds if $v\in V(h)_k$ and $k$
is maximal. If $v\ne 0$, then $o_h(\ord_X(v))$ $\leq\nu(h)-o_h(v)$, and 
$\dim_{\F_q}(V(h)/V(h)_k)\geq k\deg(h)$ for all $k\leq\lambda(h)_1$.
These inequalities become equalities when $X$ is $f$-cyclic relative to $h$. 
In the case that $X$ is $f$-cyclic relative to $h$,
then $V(h)$ is uniserial, and a uniformly distributed random vector
$v\in V(h)$ has $o_h(v)=k$ with probability
\[
  \frac{|V(h)_k| - |V(h)_{k+1}|}{|V(h)|}
  =q^{-k\deg(h)}-q^{-(k+1)\deg(h)}.
\]
Each vector $v\in V$ has a unique decomposition $v=\sum_h v_h$
where each $h$ is irreducible and $v_h$ belongs to the $h$-primary component $V(h)$ of $V$.
Thus, for a non-constant divisor $g$ of $c_X(t)$, $v$ is a $g$-witness if
and only if $v_h\not\in V(h)_1$ holds for each irreducible divisor $h$ of $g$,
or equivalently, $o_h(v)=0$ for each irreducible divisor $h$ of $g$.
This happens
with probability $\prod_{h|g}(1-q^{-\deg(h)})$, where the product is over all 
(monic) irreducible divisors $h$ of $g$.

\subsection{Is$f$Witness}\label{SS:IsfWitness}
The algorithm \IsfCyclic\ has input $(X,\eps)$, and makes repeated calls
to a deterministic subprogram \IsfWitness\ with input $(v,X,c_X(t))$, where $v$ is a
uniformly distributed random vector in $V=\F_q^{1\times n}$. 
%in the case when $|F|$ is finite, a {\it uniformly} random vector. 
Because $c_X(t)$ should be
calculated once, and not each time the subprogram \IsfWitness\ is invoked,
it is listed as an input parameter for \IsfWitness.
The algorithm \IsfWitness\ outputs \textsc{True} if $v$ is a $g$-witness
for $X$ for some non-constant divisor $g$ of $c_X(t)$, or \textsc{False} if $v$ is not
an $g$-witness for any non-constant divisor $g$ of $c_X(t)$. As the
\MeatAxe\ requires a useful certificate of $f$-cyclicity,
in the former case, \IsfWitness\ outputs a triple $(\textsc{True},u,a(t))$ where
$u\ne0$, $\ord_X(u)=a(t)$, $\gcd(a(t),c_X(t)/a(t))=1$, and
$u$ is an $a(t)$-witness.
% for {\it all} irreducible divisors $f$ of $a(t)$.

The subprogram \IsfWitness\ introduces a vector $u$ and polynomials $a,d,g$ that 
are modified in the course of the algorithm. However, each time line~5 is
executed, the relations $u=vg(X)$, $a=\ord_X(u)$,
and $d=\gcd(a,c_X(t)/a)$ always hold, see Theorem~\ref{T:correct}(a). It is
useful to note that if $d$ divides $a=\ord_X(u)$, then $\ord_X(ud(X))=a/d$.

\noindent
{\bf Algorithm.} \IsfWitness\n
{\bf Input.} \hskip10mm a non-zero vector $v\in V$; $X\in\M(n,q)$; the characteristic polynomial $c_X(t)$\n
{\bf Output.} \hskip6mm $(\textsc{True},u,a(t))$, or \textsc{False}
\begin{enumerate}
\item[1.] $u:=v$; $g(t):=1$; \hskip20.7mm \text{\# $u=vg(X)$ always holds}
\item[2.] $a(t):=\ord_X(u)$; \hskip22mm \text{\# compute the order polynomial
  of $u$ under $X$}
\item[3.] $d(t):=\gcd(a(t),c_X(t)/a(t))$; \text{\# $d$ is always
  $\gcd(a,c_X(t)/a)$}
\item[4.] $i:=1$;%$\textup{finished}:=\textsc{False};$
\item[5.] {\sc while} $i\leq \lfloor\log_2n\rfloor+2$ do
\item[6.] \hskip10mm if $d=1$ then return $(\textsc{True},u,a(t))$; fi;
\item[7.] \hskip10mm if $d=a$ then return \textsc{False}; fi; 
  \hskip7.1mm\text{\# henceforth $d\ne 1,a$ and $d$ divides $a$}
\item[8.] \hskip10mm $g:=g*d$; $u:=vg(X)$;\hskip22mm\text{\# $u:=ud(X)$
  is less efficient}
\item[9.] \hskip10mm $a:=a/d$; \hskip45.4mm\text{\# $a=\ord_X(u)=\ord_X(v)/g$
  always hold}
\item[10.] \hskip10mm $e:=\gcd(a,d)$; $d:=e*\gcd(a/e,e)$;
  \text{\# $d=\gcd(a,c_X(t)/a)$ always holds}
\item[11.] \hskip10mm $i:=i+1$; \hskip46mm\text{\# $i=$ number of times line~5 is executed}
\end{enumerate}

\begin{theorem}\label{T:correct}
Parts \textup{(a)--(e)} below prove the correctness of the
algorithm \IsfWitness. Let $c_X(t)=\prod_{h}h^{\nu(h)}$, 
where the product is over all (monic) irreducible 
divisors of $c_X(t)$. Suppose that line~$5$ is executed $s$ times, and
the values of $u, a(t), d(t)$ and $g(t)$
at the $i$th iteration of line~5 are
$u_i, a_i(t), d_i(t)$ and $g_i(t)$,
respectively. Also set $b_i:=c_X(t)/a_i(t)$.
\begin{enumerate}
\item[(a)] Then $u_i\ne 0$, $a_i=\ord_X(u_i)\neq 1$, $u_i=vg_i(X)$, and
$d_i=\gcd(a_i,b_i)$ for $1\leq i\leq s$.
\item[(b)] Set $k(h):=\nu(h)-o_h(\ord_X(v))$ for each irreducible divisor
$h$ of $c_X(t)$, and set
$r(h):=\left\lfloor\log_2\frac{\nu(h)}{k(h)}\right\rfloor+1$ when $k(h)>0$.
Then either
\begin{enumerate}
  \item[(i)] $k(h)=0$, and for $i\geq1$, $o_h(a_i)=\nu(h)$ and $o_h(d_i)=0$; or
  \item[(ii)]  $k(h)>0$, and $o_h(a_i)=o_h(d_i)=0$ for $i\geq r(h)+1$.
\end{enumerate}

\item[(c)] Set  $r:=0$ if $k(h)=0$ for all irreducibles $h$, 
and set $r:=\max\left\{r(h)\mid k(h)>0\right\}$ otherwise.
Then  $s\leq r+1\leq\lfloor\log_2 n\rfloor+2$. 
Also \IsfWitness\ returns $(\textsc{True},u_s,a_s)$ at line~$6$, or
{\sc False} at line~$7$. In either case, $\max\{1,r-1\}\leq s\leq r+1$ holds.
%Also \IsfWitness\ either returns
%$(\textsc{True},u_s,a_s)$ at line~$6$ and $r-1\leq s\leq r+1$, or it returns
%{\sc False} at line~$7$.

\item[(d)] \IsfWitness\ returns $\textsc{True}$ if and only if
$v$ is an $a$-witness for $X$ for some non-constant divisor $a$ of $c_X(t)$.
%If \IsfWitness\ returns $(\textsc{True},u_s,a_s)$, then
%$X$ is $f$-cyclic relative to every irreducible divisor of $a_s$
%and $0\ne V(a_s):=\oplus_{h|a_s}V(h)=u_s\F_q[X]\subseteq v\F_q[X]$.
%Moreover, if $v$ is an $h$-witness for $X$ for some $h$, then

\item[(e)] \IsfWitness\ returns \textsc{False} if and only if
\,$o_h(\ord_X(v))<\nu(h)$ for each irreducible polynomial $h$ such that
$X$ is $f$-cyclic relative to $h$. In particular,
\IsfWitness\ returns \textsc{False} if $X$ is uncyclic.

\item[(f)] \IsfWitness\ requires $\OO(\Mat(n)\log n)$ field operations.
\end{enumerate}
\end{theorem}

\begin{proof}
(a) We use induction on~$i$. Part~(a) holds for $i=1$
by the definitions of $u, a, d, g$ in lines 1--3 of \IsfWitness. 
Suppose inductively that the claimed relations hold for $1\leq i<s$.
As $i<s$, \IsfWitness\ does not terminate at lines 6 or 7 on the $i$th
iteration, and
it follows that $d_i\ne1, a_i$. The new values of these variables assigned 
during the $i$th iteration of lines 8--10 are 
$g_{i+1}=g_i*d_i, u_{i+1}=vg_{i+1}(X), a_{i+1}=a_i/d_i$, and
$d_{i+1}=e*\gcd(a_{i+1}/e,e)$ where $e=\gcd(a_{i+1},d_i)$. 
Since
$\ord_X(u_{i+1})=\ord_X(vg_{i+1})=$ $\ord_X(u_id_i)=a_i/d_i\ne 1$, 
it follows that $u_{i+1}\ne 0$ and $a_{i+1}=\ord_X(u_{i+1})$.
By definition $b_i=c_X/a_i$, and hence
$b_{i+1}=c_X/a_{i+1}=b_i*a_i/a_{i+1}=b_i*d_i$. 
Finally, we must prove that
$d_{i+1}=\gcd(a_{i+1},b_{i+1})$. Now $d_i=\gcd(a_i,b_i)$ implies that
$\gcd(a_i/d_i,b_i/d_i)=1$, that is, $\gcd(a_{i+1},b_i/d_i)=1$.
Similarly $e=\gcd(a_{i+1},d_i)$ implies that $\gcd(a_{i+1}/e,d_i/e)=1$. To
complete the inductive proof of part (a) we show that $\gcd(a_{i+1},b_{i+1})$
is equal to $e*\gcd(a_{i+1}/e,e)$, which is $d_{i+1}$:
\begin{align*}
  \gcd(a_{i+1},b_{i+1})
  &=\gcd\left(a_{i+1},b_i*d_i\right)&&\mbox{(since}\ b_{i+1}=b_i*d_i)\\
  &=\gcd(a_{i+1},\frac{b_i}{d_i}*d_i^2)\\
  &=\gcd(a_{i+1},d_i^2)&&\mbox{(since}\ \gcd(a_{i+1},\frac{b_i}{d_i})=1)\\
  &=e*\gcd(\frac{a_{i+1}}{e},\frac{d_i}{e}*\frac{d_i}{e}*e)\\
  &=e*\gcd(\frac{a_{i+1}}{e},e)&&\mbox{(since}\ \gcd(\frac{a_{i+1}}{e},\frac{d_i}{e})=1).
\end{align*}
 
(b)  
Before proving part~(b) we shall prove (\ref{E:claim}), (\ref{E:sln})
and (\ref{E:ohd}) below.
Note that $o_h(c_X)=\nu(h)$ $=o_h(a_i)+o_h(b_i)$ for all $i$.
It follows from $k(h)=\nu(h)-o_h(\ord_X(v))$ and $a_1=\ord_X(v)$, that $o_h(b_1)=k(h)$. 
We first prove
\begin{equation}\label{E:claim}
  o_h(b_{i+1})=\left\lbrace \begin{array}{ll}
  0&\mbox{if $o_h(b_i)=0$,}\\
  2\,o_h(b_i)&\mbox{if $0<o_h(b_i)\leq\nu(h)/2$,}\\
  \nu(h)&\mbox{if $o_h(b_i)> \nu(h)/2$.}\\
                       \end{array}\right.
\end{equation}
Suppose first that $o_h(b_i)=0$. Then $o_h(d_i)=o_h(\gcd(a_i,b_i))=0$ and hence
$o_h(b_{i+1})$ equals $o_h(b_id_i)=o_h(b_i)=0$. This establishes the first
part of (\ref{E:claim}).
Next suppose that $0<o_h(b_i)\leq\nu(h)/2$. Then   $o_h(a_i)\geq o_h(b_i)$,
and so $o_h(d_i)=o_h(\gcd(a_i,b_i))=o_h(b_i)$, which implies that
$o_h(b_{i+1})=o_h(b_id_i)=2o_h(b_i)$. 
Finally, suppose that $o_h(b_i)> \nu(h)/2$. Then  $o_h(a_i)< o_h(b_i)$,  
and so $o_h(d_i)=o_h(\gcd(a_i,b_i))$ $=o_h(a_i)$, which implies that
$o_h(b_{i+1})=o_h(b_id_i)=o_h(b_i)+o_h(d_i)=o_h(b_i)+o_h(a_i)=\nu(h)$.
Thus (\ref{E:claim}) is proved. 

It is useful to solve the recurrence relation~(\ref{E:claim}). We next prove
that
\begin{equation}\label{E:sln}
  o_h(b_i)=\left\lbrace \begin{array}{ll}
  0&\mbox{if $k(h)=0$,}\\
  2^{i-1}k(h)&\mbox{if $k(h)>0$ and $1\leq i\leq r(h)$,}\\
  \nu(h)&\mbox{if $k(h)>0$ and $i>r(h)$.}\\
                       \end{array}\right. 
\end{equation}
Certainly if $k(h)=0$ then since $o_h(b_1)=k(h)$ (as we noted above), 
it follows that $o_h(b_1)=0$. By (\ref{E:claim}), we have $o_h(b_i)=0$
for all $i$. This establishes the first part of (\ref{E:sln}).
Suppose now that $k(h)>0$. We next prove (\ref{E:sln}) for 
$1\leq i\leq r(h)$ using induction on~$i$. The claim in (\ref{E:sln})
is true when $i=1$ as $o_h(b_1)=k(h)$. Suppose that $1\leq i< r(h)$ and 
$o_h(b_i)=2^{i-1}k(h)$. Then $i+1\leq r(h)$ and it 
follows from the definition of $r(h)$ that  $2^i k(h)\leq \nu(h)$,
and hence that $0<o_h(b_i)\leq \nu(h)/2$. Hence by  (\ref{E:claim})
we have $o_h(b_{i+1})=2^i k(h)$. Thus  (\ref{E:sln}) holds by induction
for $1\leq i\leq r(h)$. In particular $o_h(b_{r(h)})=2^{r(h)-1}k(h)$. 
Now by the definition of $r(h)$ we have $2^{r(h)}>\nu(h)/k(h)$, 
and hence $o_h(b_{r(h)})>\nu(h)/2$. Hence, by (\ref{E:claim}),
$o_h(b_{r(h)+1})=\nu(h)$, and by repeated applications of 
(\ref{E:claim}), $o_h(b_i)=\nu(h)$ for all $i>r(h)$.
Thus  (\ref{E:sln}) is proved.

Equation (\ref{E:sln}) may be used to compute $o_h(d_i)$.
In this paragraph we prove that
\begin{equation}\label{E:ohd}
  o_h(d_i)=\left\lbrace \begin{array}{ll}
  0&\mbox{if $k(h)=0$ or $i>r(h)$,}\\
  2^{i-1}k(h)&\mbox{if $k(h)>0$ and $1\leq i< r(h)$,}\\
  \nu(h)-2^{r(h)-1}k(h)&\mbox{if $k(h)>0$ and $i=r(h)$.}
                       \end{array}\right. 
\end{equation}
Part~(a) gives $o_h(d_i)=o_h(\gcd(a_i,b_i))=\min(o_h(a_i),o_h(b_i))$. Thus if
$o_h(b_i)$ equals 0 or $\nu(h)$, then $o_h(d_i)=0$. This establishes the
first part of (\ref{E:ohd}). Consider the second part, and assume that
$k(h)>0$ and $1\leq i< r(h)$. It follows from the previous paragraph that
$0<o_h(b_i)\leq \nu(h)/2$. Thus $o_h(d_i)=o_h(b_i)=2^{i-1}k(h)$ by (\ref{E:sln}).
Finally, suppose that $k(h)>0$ and $i=r(h)$. By the previous paragraph
$o_h(b_{r(h)})>\nu(h)/2$ and so
$o_h(d_{r(h)})=o_h(a_{r(h)})=\nu(h)-2^{r(h)-1}k(h)$ by~(\ref{E:claim}).
This proves (\ref{E:ohd}).

The proof of part~(b) is now simple. If $k(h)=0$, then
$o_h(b_1)=0$ and $o_h(a_1)=\nu(h)$ hold. Thus part~(i) follows from
(\ref{E:sln}) and (\ref{E:ohd}).
On the other hand, if $k(h)>0$, then (\ref{E:sln}) and (\ref{E:ohd}) imply
that $o_h(b_i)=\nu(h)$ and $o_h(d_i)=0$ for $i\geq r(h)+1$.
Thus part~(ii) holds.

(c)
We first prove that $s\leq r+1$. 
Note that $d_i=1$ is equivalent to $o_h(d_i)=0$ for all~$h$.
Suppose that the number, $s$, of times that line~5 is executed satisfies
$s\geq r+1$. Then it follows from part~(b) that $d_{r+1}=1$, and hence
that \IsfWitness\ terminates on executing line 6, and $s=r+1$.
Thus $s\leq r+1$. (Note that if $d_i=a_i$ for some $i<r+1$ then
\IsfWitness\ terminates at line~7, and $s<r+1$.)
Thus $s\leq r+1\leq \lfloor\log_2 n\rfloor+2$ where the last
inequality follows as $r=r(h)$ for some $h$ with $k(h)\geq1$ and $\nu(h)\leq n$.
%Thus, if the algorithm enters iteration $r+1$ of the while loop, then 
%it follows from part~(b) that $d_{r+1}=1$, and hence that the algorithm
%terminates on executing line 6 in this iteration.
%Thus the number, $s$, of times line~5 is executed satisfies $s\leq r+1$.
%In summary, 
%if $s=r+1$, then \IsfWitness\ terminates during iteration $r+1$ when
%executing the conditional statement at line~6.
%However, if $d_i=a_i$ for some $i<r+1$ then
%\IsfWitness\ terminates earlier at line~7.
%Thus $s\leq r+1\leq \lfloor\log_2 n\rfloor+2$ where the last
%inequality follows as $r=r(h)$ for some $h$ with $k(h)\geq1$ and $\nu(h)\leq n$.
This proves the second and third sentences of part~(c).
To prove the last sentence we must show that $r-1\leq s$
(as $1\leq s$ is clear). This is certainly true if $r\leq2$.
Suppose now that $r\geq3$. Fix an irreducible
polynomial~$h$ such that $r=r(h)$. Then $k(h)>0$. Showing that $s\not<r-1$
is equivalent to showing that \IsfWitness\ does not terminate during
iteration $i$ when $i<r-1$.
This is equivalent to proving $d_i\ne a_i$ and $d_i\ne1$ holds for $i<r-1$
which, in turn, is proved by showing $0<o_h(d_i)<o_h(a_i)$ for $i<r(h)-1$.
The inequalities $0<o_h(d_i)$ with $i<r(h)-1$ hold
by~(\ref{E:ohd}). It follows from (\ref{E:ohd}) and (\ref{E:sln}) that
$o_h(d_i)=2^{i-1}k(h)$ and $o_h(a_i)=\nu(h)-2^{i-1}k(h)$.
However, $i<r(h)-1$ implies
$i<\lfloor\log_2\frac{\nu(h)}{k(h)}\rfloor$,
which implies $2^ik(h)<\nu(h)$, and hence $o_h(d_i)<o_h(a_i)$.
Thus $r-1\leq s\leq r+1$ and part~(c) is proved. 

(d) Consider the forward implication. Suppose that \IsfWitness\ returns
$(\textsc{True},u_s,a_s)$. Then $d_s=1$, and by part (a), 
$\gcd(a_s,c_X/a_s)=1$ and $a_s\ne1$. Thus 
$a_s=\prod_{h|a_s}h^{\nu(h)}\ne 1$, and the Chinese Remainder Theorem gives
\[
  V(a_s):=\bigoplus_{h|a_s}V(h)=u_s\F_q[X]\subseteq v\F_q[X]
  \text{ as }V(a_s)\cong\F_q[t]/(a_s)\text{ and }V(h)\cong\F_q[t]/(h^{\nu(h)}).
\]
This proves that $v$ is an $a_s$-witness for $X$, and $X$ is $f$-cyclic relative
to each irreducible divisor $h$ of $a_s$. Now consider the reverse implication.
Suppose that $v$ is an $a$-witness for $X$ for some non-constant divisor 
$a$ of $c_X$. Then by the definition of an $a$-witness in Subsection~\ref{witness}, 
for each irreducible divisor $h$ of $a$, $V(h)\subseteq v\F_q[X]$, and 
it follows that $k(h)=o_h(b_1)=0$. Thus by
(\ref{E:sln}), $o_h(b_i)=0$ and $o_h(a_i)=\nu(h)>0$ for all $i$. Hence
$0=o_h(d_i)<o_h(a_i)$ for all $i$, and the conditional line~7
of \IsfWitness\ is never executed. It now follows from part~(c) that
\IsfWitness\ returns $\textsc{True}$. This proves part~(d). 

(e)~
Suppose that \IsfWitness\ returns {\sc False}, and 
let $h$ be an irreducible polynomial such that $X$ is $f$-cyclic relative to $h$. 
If $o_h(\ord_X(v))=\nu(h)$, then $k(h)=o_h(b_1)=0$, and the 
argument of the previous paragraph gives that \IsfWitness\ 
returns {\sc True}, which is a contradiction. Hence  
$o_h(\ord_X(v))<\nu(h)$. Conversely suppose that
$o_h(\ord_X(v))<\nu(h)$, for each irreducible 
polynomial $h$ such that $X$ is $f$-cyclic relative to $h$.
Then for each such $h$, $v$ is not an $h$-witness for $X$,
and it follows from the previous paragraph that
\IsfWitness\ does not return {\sc True}. Since 
\IsfWitness\ returns an answer by part (c), it must return {\sc False}.
This proves the first
sentence of part~(e). The second sentence is an immediate consequence of the first.

%Note first that the following conditions are equivalent:
%(1) $v$ is an $h$-witness for~$X$, (2) $k(h)=0$,
%(3) $0\ne V(h)\subseteq v\F_q[X]$, and (4) $o_h(\ord_X(v))=\nu(h)>0$.

(f) The cost of multiplication, division,
or finding the greatest common divisor of two polynomials, each of 
degree at most $n$, is $\OO(n^2)$ field operations. 
As $c_X(t)$ is an input parameter to \IsfWitness,
line~3 has cost $\OO(n^2)$. 
Computing $\ord_X(v)$ in line~2 has cost
$\OO(\Mat(n)\log n)$ by \cite[Theorem~6.2.1(b)]{Ambrose}.
When computing $\ord_X(v)$, one uses ``fast spinning'' to calculate an
$n\times n$ matrix $Y$ with rows $v,vX,\dots,vX^{n-1}$.
We must remember $Y$ in order to compute, for each $i$, the vector $u_i$ in line~8.
If $g_i(t)=\sum_{j=0}^{n-1}g_{ij}t^j$, then $u_i=(g_{i0},g_{i1},\dots,g_{i,n-1})Y$.
Thus the cost of lines~8, 9, 10 in the $i$th iteration of the {\sc while} loop is $\OO(n^2)$.
By Theorem~\ref{T:correct}(c) the {\sc while} loop is executed at most $\lfloor\log_2n\rfloor +2$ times.
Thus the total cost of running the {\sc while} loop is $\OO(n^2\log n)$.  Since $\Mat(n)$ is at
least $\OO(n^2)$, it follows that \IsfWitness\ requires at most 
$\OO(\Mat(n)\log n)$ field operations.
\end{proof}

\begin{remarks}
(a) If we use standard algorithms for vector-matrix operations, then
an upper bound for the cost of \IsfWitness\ is $\OO(n^3)$. 
For example, at line $2$ 
the cost of finding $\ord_X(v)$ if one uses standard vector-matrix 
arithmetic is $\OO(n^3)$, (see for example,
\cite[Proposition 4.9]{NeunP}). Similarly, at line~$9$ 
we may replace $u:=vg(X)$ by $u:=ud(X)$. Using the notation of
Theorem~\ref{T:correct}, the sum of the degrees of the polynomials
$d_1,\dots,d_s$ is at most $n$. Hence the cost of computing $u_1,\dots,u_s$ is
at most $\OO(n^3)$.
The complexity bound $\OO(n^3)$ follows from these observations.

(b) The algorithm \IsfWitness\ may be varied as follows.
In essence \IsfWitness\ seeks a divisor $a$ of $\ord_X(v)$ of
\emph{maximal degree} satisfying $\gcd(a,c_X(t)/a)=1$.
Although it is straightforward to calculate $a$ from the
factorisation of $c_X(t)$ as a product of irreducibles, it is also possible
to calculate $a$ using only gcd's and $p$th roots where $p$ is the
characteristic of $\F_q$. We omit the precise details, but the computation
of square-free factorizations, see \cite[Algorithm~3.4.2]{Cohen}, plays an
important role.
\end{remarks}

\subsection{Algorithm Is$f$Cyclic}\hfill\newline
\noindent
{\bf Algorithm.} \IsfCyclic\n
{\bf Input.} \hskip10mm a (non-zero) matrix $X\in\M(n,q)$; a positive
real number $\eps<1$\n
{\bf Output.} \hskip6mm $(\textsc{True},u,a(t))$, or \textsc{False}
\begin{enumerate}
\item[1.] $m:=\left\lceil\frac{\log(\eps^{-1})}{\log q}\right\rceil;$ \# $m$ is
  the maximum number of random vectors tested
\item[2.] $c:=c_X(t);$\hskip10.9mm\# compute the characteristic polynomial of $X$
\item[3.] $i:=1;$\hskip18.4mm\# $i$ counts the number of random vectors chosen
\item[4.] {\sc while} $i\leq m$ do
\item[5.] \hskip10mm $v:=$ a (uniformly) random vector in $\F_q^{1\times n}$; if $v=0$ then continue; fi;
\item[6.] \hskip10mm  $\text{output}:=\IsfWitness(v,X,c)$;
\item[7.] \hskip10mm if $\text{output}\ne\textsc{False}$ then return $\text{output}$; fi; $i:=i+1$;
\item[8.] return \textsc{False};\hskip4.5mm\# probability of failure given that
$X$ is $f$-cyclic is at most $\eps$
\end{enumerate}

Recall that, for an $f$-cyclic matrix $X\in\M(n,q)$ and a non-constant divisor 
$a$ of $c_X$, a vector $v\in\F_q^n$ is an $a$-witness for $X$ if
$v\F_q[X]$ contains $V(a)$.

\begin{theorem}\label{T:cost}
\IsfCyclic\ is a one-sided Monte Carlo algorithm for which, a given  
matrix $X\in\M(n,q)$ and positive real number $\eps<1$, the following hold.
\begin{enumerate}
 \item[(a)] If $X$ is $f$-cyclic, then \IsfCyclic\ returns
$(\textsc{True},u,a)$ with probability at least $1-\eps$, where $a$ is 
a non-constant divisor of $c_X$, and $u$ is an $a$-witness for $X$.
\item[(b)] If $X$ is uncyclic, then \IsfCyclic\  returns {\sc False} with probability $1$.
\end{enumerate}
The number of field operations required
by \IsfCyclic\ is %at most 
$\OO(\frac{\log(\eps^{-1})}{\log q}(\xi_{q,n}+\Mat(n)\log n))$, where
$\xi_{q,n}$ is 
an upper bound for the cost of constructing a uniformly distributed random 
vector in $\F_q^n$. 
%The cost of algorithm \IsfCyclic\ is
%$\OO(\frac{\log(\eps^{-1})\Mat(n)\log n}{\log q})$
%field operations where $\Mat(n)$ is an upper bound for the number of field
%operations to multiply two matrices in $\M(n,q)$. The probability that
%\IsfCyclic\ returns \textsc{Fail} when $X$ is $f$-cyclic (for some~$f$) is 
%at most $\eps$.
\end{theorem}

\begin{proof}
For $m:=\left\lceil\frac{\log(\eps^{-1})}{\log q}\right\rceil$, we have
$q^m\geq\eps^{-1}$.
Let $X$ be an $f$-cyclic matrix relative to at least one irreducible, say $h$. 
Suppose that \IsfCyclic\ returns {\sc False}.
Then \IsfWitness\ returns {\sc False} for
$m$~independent uniformly distributed random vectors of~$V$.
By the remarks preceding Subsection~\ref{SS:IsfWitness}, this
happens with probability at most
$q^{-m\deg(h)}\leq q^{-m}\leq \eps$, since $\deg(h)\geq1$.
If \IsfCyclic\ does not return {\sc False}, then at least one of the
runs of \IsfWitness\ has output ({\sc True}, $u,a)$, and this is then
returned by \IsfCyclic\ at line~7. This proves part~(a).

Now suppose that \IsfCyclic\ is has an uncyclic matrix $X$ 
as input. Then
by Theorem~\ref{T:correct}(e), each run of \IsfWitness\ returns {\sc False}, 
and hence \IsfCyclic\ returns {\sc False}. This proves part~(b).

The only situation in which the output of \IsfCyclic\ is incorrect is if
the input matrix $X$ is $f$-cyclic and  \IsfCyclic\  returns {\sc False}.
We have shown that this probability of this happening, given that $X$ 
is $f$-cyclic, at most than $\eps$. Thus \IsfCyclic\ is a one-sided 
Monte Carlo algorithm.

Finally, we estimate the cost.
Computing $c_X(t)$ in line~2
of \IsfCyclic\ requires at most $\OO(\Mat(n)\log n)$ field operations, see \cite{Ambrose}.
%We alert the reader to the claim in \cite{BCS} that $c_X(t)$
%can be computed with cost $\OO(\Mat(n))$. This claim is made
%on the basis of an incorrect complexity analysis of the
%Keller-Gehrig algorithm. 
By Theorem~\ref{T:correct}(f), the total cost of $m$ iterations of the
{\sc while} loop of \IsfCyclic\ is $\OO(m\Mat(n)\log n)$ plus the cost of constructing
$m$ uniformly distributed random vectors from $\F_q^n$.
\end{proof}

\section*{Acknowledgements}

The second author acknowledges support of an Australian Research Council
Federation Fellowship.
%This research forms part of a ARC Discovery Project grant DP0879134.
Both authors acknowledge the ARC Discovery Project grant
DP0879134 which supported the first author's visit to The University
of Western Australia.

\end{document}